\documentclass[12pt,leqno,twoside]{amsart}

\usepackage{amssymb,amsmath,amsthm,soul,color}

\usepackage[noadjust]{cite}

\usepackage{t1enc}

\usepackage[utf8]{inputenc}
\usepackage{indentfirst}

\usepackage{todonotes}

\usepackage[left=2.5cm, right=2.5cm, top=2cm, bottom=2cm]{geometry}

\usepackage{url}

\usepackage{mathrsfs}

\usepackage{enumitem,graphicx}
\usepackage[colorlinks=true,urlcolor=blue,
citecolor=red,linkcolor=blue,linktocpage,pdfpagelabels,
bookmarksnumbered,bookmarksopen]{hyperref}

\usepackage[metapost]{mfpic}
\usepackage[normalem]{ulem}

\usepackage{lmodern} 

\newtheorem{Th}{Theorem}[section]
\newtheorem{Prop}[Th]{Proposition}
\newtheorem{Lem}[Th]{Lemma}
\newtheorem{Cor}[Th]{Corollary}
\newtheorem{Def}[Th]{Definition}
\newtheorem{Rem}[Th]{Remark}
\newtheorem{Ex}[Th]{Example}

\newcommand{\vp}{\varphi}

\newcommand{\eps}{\varepsilon}

\newcommand{\R}{\mathbb{R}}

\newcommand{\Z}{\mathbb{Z}}
\newcommand{\N}{\mathbb{N}}
\newcommand{\bR}{\mathbb{R}} 

\newcommand{\bZ}{\mathbb{Z}}
\newcommand{\bN}{\mathbb{N}}

\newcommand{\cA}{{\mathcal A}}
\newcommand{\cB}{{\mathcal B}}
\newcommand{\cC}{{\mathcal C}}
\newcommand{\cD}{{\mathcal D}}
\newcommand{\cE}{{\mathcal E}}
\newcommand{\cF}{{\mathcal F}}
\newcommand{\cG}{{\mathcal G}}
\newcommand{\cH}{{\mathcal H}}
\newcommand{\cI}{{\mathcal I}}
\newcommand{\cJ}{{\mathcal J}}

\newcommand{\cL}{{\mathcal L}}
\newcommand{\cM}{{\mathcal M}}
\newcommand{\cN}{{\mathcal N}}
\newcommand{\cO}{{\mathcal O}}
\newcommand{\cP}{{\mathcal P}}

\newcommand{\cS}{{\mathcal S}}

\newcommand{\cU}{{\mathcal U}}

\newcommand\J{\mathcal{J}}

\newcommand{\weakto}{\rightharpoonup}

\newcommand{\tu}{\widetilde{u}}

\renewcommand\J{\mathcal{J}}
\newcommand{\lJ}[2]{\cJ_{#1}^{#2}}

\numberwithin{equation}{section}

 \newcommand{\curl}{\nabla \times}
 
 \renewcommand{\div}{\mathrm{div}\,}

\newcommand{\crit}{\mathrm{crit}}
\newcommand{\critJ}{\crit(\cJ)}
\newcommand{\dd}{\mathrm{d}}

\newcommand{\lowlevel}{-\alpha}

\begin{document}

\title[Multiplicity theorem for functionals with sign-changing nonlinearity]{Multiplicity of critical orbits to nonlinear, strongly indefinite functionals with sign-changing nonlinear part}

\author[F. Bernini]{Federico Bernini}
    \address[F. Bernini]{\newline\indent
	Dipartimento di Matematica e Applicazioni
	\newline\indent
	Università degli Studi di Milano-Bicocca
	\newline\indent
	Via R. Cozzi, 55
	\newline\indent
	I-20125 Milan, Italy.}
    \email{\href{mailto:federico.bernini@unimib.it}{federico.bernini@unimib.it}
}

\author[B. Bieganowski]{Bartosz Bieganowski}
	\address[B. Bieganowski]{
	\newline\indent Faculty of Mathematics, Informatics and Mechanics,
	\newline\indent University of Warsaw,
	\newline\indent ul. Banacha 2, 02-097 Warsaw, Poland}	
	\email{\href{mailto:bartoszb@mimuw.edu.pl}{bartoszb@mimuw.edu.pl}}	
	
\author[D. Strzelecki]{Daniel Strzelecki}
	\address[D. Strzelecki]{
	\newline\indent Faculty of Mathematics, Informatics and Mechanics,
	\newline\indent University of Warsaw,
	\newline\indent ul. Banacha 2, 02-097 Warsaw, Poland}	
	\email{\href{mailto:dstrzelecki@mimuw.edu.pl}{dstrzelecki@mimuw.edu.pl}}	

\pagestyle{myheadings} \markboth{\underline{F. Bernini, B. Bieganowski, D. Strzelecki}}{
		\underline{TBA}}

\newcommand{\spann}{\mathrm{span}\,}

\newcommand{\triple}[1]{{\left\vert\kern-0.25ex\left\vert\kern-0.25ex\left\vert #1 
\right\vert\kern-0.25ex\right\vert\kern-0.25ex\right\vert}}

\begin{abstract} 
We present an abstract critical point theorem about the existence of infinitely many critical orbits to strongly indefinite functionals with a sign-changing nonlinear part defined on a dislocation space with a discrete group action. We apply the abstract result to a Schr\"odinger equation
$$
-\Delta u + V(x) u = f(u) - \lambda g(u)
$$
with $0$ in the spectral gap of the Schr\"odinger operator $-\Delta + V(x)$, that appears in nonlinear optics. We also consider equations with singular potentials arising from the study of cylindrically symmetric, electromagnetic waves to the system of Maxwell equations.

\medskip

   
\end{abstract}

\maketitle

\section{Introduction}

In this paper, we are interested in an abstract critical point theory that allows us to study the multiplicity of critical points of strongly indefinite functionals. Consider a real Hilbert space $(X, \langle \cdot, \cdot \rangle)$ and a nonlinear, $\cC^1$-functional $\cJ : X \rightarrow \R$. In variational methods, looking for solutions to certain partial differential equations reduces to finding \textit{nontrivial critical points} of $\cJ$, namely points $u \in X \setminus \{0\}$ with $\cJ'(u) = 0$. 

To introduce the notion of strongly indefinite problems, assume that $\cJ$ is of $\cC^2$ class and that there is an orthogonal splitting $X = X^+ \oplus X^-$ such that the second variation $\cJ''(0)[u,u]$ is positive definite on $X^+$ and negative definite on $X^-$. If $X^- =\{0\}$ then we say that $\cJ$ is \textit{positive definite}; otherwise we say that $\cJ$ is \textit{strongly indefinite}. 

Existence of nontrivial critical points of positive definite functionals can be obtained, under some geometrical assumptions, via the mountain pass theorem introduced by Ambrosetti and Rabinowitz \cite{AmbrRab} or via the Nehari manifold method proposed in \cite{Nehari} that is closely related to the Poho\v{z}aev's fibering method \cite{PohozaevFibering1, PohozaevFibering2}. For strongly indefinite problems, the mentioned techniques have been generalized. An important contribution is due to Rabinowitz \cite{Rab78} in the case when one of the subspaces $X^+$, $X^-$ is finite-dimensional and is known as the \textit{linking theorem}. It was later generalized by Kryszewski and Szulkin \cite{KS} to the case where both subspaces have infinite dimensions. The Nehari manifold approach was extended by Pankov \cite{Pankov}, and later applied by Szulkin and Weth \cite{SzW} to the setting where both subspaces are infinite-dimensional. 
Suppose that the functional $\cJ$ has the following form
$$
\cJ(u) = \frac12 \|u^+\|^2 - \frac12 \|u^-\|^2 - \cI(u),
$$
where $u = u^+ + u^- \in X$ and and the nonlinear functional $\cI$ exhibits super-quadratic growth at infinity. In applications, the quadratic part of $\cJ$ is related to a differential operator $L$ and $\cI$ to the nonlinearity $f$ of a partial differential equation $Lu = f(u)$. We emphasize that in all the mentioned results it is required that $\cI(u) \geq 0$.

Recently, there have been contributions that allow $\cI$ to change its sign in $X$. For the positive definite case, we refer to \cite{BiegMed}, where the Nehari manifold method has been adapted, and to \cite{Bi} for the mountain pass approach. In the strongly indefinite case, we refer the reader to \cite{BB, CW}.

In the case where $\cJ$ enjoys symmetry properties, e.g. if it is even, one can expect to show also the multiplicity of solutions (more precisely: existence of infinitely many solutions). a particularly powerful tool is the fountain theorem provided by Bartsch \cite{Bartsch}, applied already in many contexts, see e.g. \cite{BC1}. Another possible approach is to use the linking geometry and the notion of the so-called \textit{(PS)-attractors} \cite{KS, BartschDing} or the Nehari-Pankov manifold \cite{SzW, dePaiva}. Recently, under some convexity assumptions on $\cI$, it has been shown that the strongly indefinite problem may be reduced to the positive definite one, obtaining the existence and multiplicity of solutions \cite{MSS, BiegMed2}. We remark again that in all the mentioned papers, the assumption $\cI(u) \geq 0$ is crucial.

In the case of a sign-changing $\cI$, not much is known so far. In the positive definite setting, the approaches developed in \cite{SzW, dePaiva} have been adapted to sign-changing $\cI$ in \cite{BB_JMAA, Bi}. For the strongly indefinite case, there is a very recent contribution \cite{Gu}, where the authors established the multiplicity of critical points for a strongly indefinite functional corresponding to the stationary Schr\"odinger equation with sign-changing nonlinearity. 

Usually, differential equations admit certain symmetries. If $G$ is a group acting on $X$ and $X^+$, $X^-$ are $G$-invariant, and $\cJ$ is a $G$-invariant functional, then for every critical point $u \in X$, the whole \textit{orbit} $\cO(u)$ of $u$ under the action of $G$ consists of critical points. Hence, in the mentioned results, the authors have shown the existence of infinitely many \textit{critical orbits} of $\cJ$. Two solutions whose orbits are disjoint are then called  \textit{geometrically distinct} (see Section \ref{sect:funct} for the formal definition).

In the present paper, we are interested in providing an abstract critical point theorem about the existence of infinitely many critical orbits of a given functional $\cJ$ under the action of a group $G$. The theorem is stated in abstract \textit{dislocation spaces} introduced by Tintarev and Fieseler \cite{Tintarev2007}. 

The paper is organized as follows. In Section \ref{sect:funct} we introduce the notion of dislocation spaces and the functional setting. The main theorem (Theorem \ref{T:mainTheorem}) is stated and proved in Section \ref{sect:abstr}. Sections \ref{sect:appl1} and \ref{sect:appl2} are devoted to applications, where we show the multiplicity of solutions to strongly indefinite Schr\"odinger-type equations with sign-changing nonlinearities that appear in nonlinear optics and when one looks for time-harmonic electromagnetic waves to the system of Maxwell equations.

\section{Functional setting}
\label{sect:funct}

Let $(X, \langle \cdot, \cdot \rangle)$ be a separable, real Hilbert space and $\cJ:X\to\bR$ be a $\cC^1$-class nonlinear functional. We recall that, due to the Riesz representation theorem, for any $u \in X$ there exists $\nabla \cJ(u) \in X$ such that $\langle \nabla \cJ(u), v \rangle =\cJ'(u)(v)$ for all $v\in X$.

\begin{Def}
We say that a sequence $(u_n)\subset X$ is a Palais-Smale sequence (or a $(PS)$-sequence) for $\cJ$ if
\[
(\cJ(u_n))\text{ is bounded}\quad\text{and}\quad \cJ'(u_n)\to 0\text{ in }X^*.
\]
We shall say that $(u_n)$ is a $(PS)_c$-sequence if it is a Palais-Smale sequence and additionally $\cJ(u_n)\to c$.
\end{Def} 

Moreover we recall that $\cJ'$ is \textit{sequentially weak-to-weak* continuous} if 
\[
u_n\weakto u \, \Rightarrow \, \forall {v\in X} \ \langle \nabla\cJ(u_n) ,v\rangle \to \langle \nabla\cJ(u),v\rangle.
\]
Clearly, this implies that any weak limit of the $(PS)$-sequence is a critical point of $\cJ$.

Suppose that there is an orthogonal splitting $X = X^+ \oplus X^-$ and for $u = u^+ + u^-$ with $u^\pm \in X^\pm$ we have
$$
\|u\|^2 = \|u^+\|^2 + \|u^-\|^2,
$$
where $\| \cdot \|$ is the norm induced by the scalar product $\langle \cdot, \cdot \rangle$. Furthermore, suppose that $\cJ$ is of the following form
\begin{equation}\label{E:introJ}
   \cJ(u)=\frac 12 \|u^+\|^2-\frac12\|u^-\|^2-\cI(u),\quad u=u^+ + u^-\in X^+\oplus X^-,
\end{equation}
where $\cI:X \to \R$ is a nonlinear functional. 

We make use the following notation:
\[
\cJ_{\alpha}=\{u\in X\,:\, \alpha<\cJ(u)\},\quad
\cJ^{\beta}=\{u\in X\,:\, \cJ(u)\leq \beta\},\quad
\cJ_{\alpha}^{\beta}=\cJ_{\alpha}\cap\cJ^{\beta}.
\]

\subsection{Dislocation spaces and profile decompositions}
From now on, we assume that on $X$ there is given a unitary action of the group $G$, i.e. there is a group homomorphism $T:G\to GL(X)$ such that $T_g=T(g)$ is a unitary operator. For simplicity, we will denote the operator $T_g$ by $g$ if no confusion can arise. 

\begin{Def}
We say that the functional $\cJ:X\to\bR$ is $G$-invariant if for any $g\in G$ and $u\in X$, we have $\cJ(gu)=\cJ(u)$. We say that a subset $A \subset X$ is $G$-invariant if for any $g \in G$ and $u \in A$, we have $gu \in A$.
\end{Def}

\begin{Rem}
If $\cJ:X\to\bR$ is $G$-invariant then the gradient $\nabla \cJ:X\to X$ is $G$-equivariant i.e. for any $g\in G$ and $u\in X$ there holds $\nabla\cJ(gu)=g\nabla\cJ(u)$. In particular, the set of critical points is $G$-invariant.
\end{Rem}

Observe that for $G$-invariant functional $\cJ$ the sets $\cJ_{\alpha}$, $\cJ^{\beta}$ are $G$-invariant.

Hereafter, we assume that the spaces $X^+$ and $X^-$ are $G$-invariant. We will introduce the notion of a dislocation space based on \cite{Tintarev2007, Tintarev2020} in the setting we deal with. Note that this definition could be given for a more general set of linear operators.
 
\begin{Def}
Let $(u_n) \subset X$, $u \in X$. We say that $u_n \stackrel{G}{\rightharpoonup} u$ if for all $\varphi \in X$
$$
\lim_{n \to \infty} \sup_{g \in G} \langle u_n - u, g \varphi \rangle = 0.
$$
\end{Def}

\begin{Def}
We say that a sequence $(g_n)\subset GL(X)$ converges strongly (resp. weakly) to $g$ if we have $g_nu\to gu$ (resp. $g_nu{\rightharpoonup} gu$) for all $u\in X$. We will use the notation $g_n{\rightharpoonup} g$ in the case of weak convergence.
\end{Def}

Note that the notion introduced above is weaker than convergence and weak-convergence in $X^*$.

\begin{Def}
We say that the pair $(X,G)$ is a dislocation space if for any sequences $(u_n)\subset X$ and $(g_n)\subset G$ the following holds
\[
g_n \not\rightharpoonup 0, \, u_n\weakto 0 \, \Rightarrow g_nu_n\weakto 0 \text{ up to a subsequence.}
\]
\end{Def}

From now on, until the end of the section, $(X,G)$ is a dislocation space. In dislocation spaces, one can show the following, general profile decomposition for bounded sequences.

\begin{Th}[\cite{Tintarev2007}, Theorem 3.1]\label{T:tintarev}
Let $(u_n) \subset X$ be a bounded sequence. Then there are $K \in \{0,1, 2, \ldots \} \cup \{\infty\}$, $w^k \in X \setminus \{0\}$, $g_n^k \in G$ for $1 \leq k < K+1$ such that $g_n^1 = id$ (if $K \geq 1$) and
\begin{align*}
    (g_n^k)^{-1} u_n &\weakto w^k, \\
    g_n^k (g_n^l)^{-1} &\weakto 0 \mbox{ for } k \neq l, \\
    \sum_{k=1}^K \| w^k \|^2 &\leq \limsup_{n\to\infty} \|u_n\|^2, \\
    u_n - \sum_{k=1}^K g_n^k w^k &\stackrel{G}{\rightharpoonup} 0.
\end{align*}
\end{Th}
Note that if $K = 0$, we have $u_n \stackrel{G}{\weakto} 0$. Otherwise $K \geq 1$.

Observe that, thanks to the $G$-invariance of $X^\pm$ and their mutual orthogonality, we obtain the following fact.

\begin{Prop}\label{P:D-convergence_of_projections}
    If $u_n \stackrel{G}{\rightharpoonup} 0$, then $u_n^{\pm}\stackrel{G}{\rightharpoonup} 0$.
\end{Prop}
\begin{proof}
    Let $\varphi = \varphi^+ + \varphi^- \in X^+ \oplus X^-$. By the orthogonality of $X^+$ and $X^-$ we have
    \[
    \begin{split}
    \langle u_n^+,g\varphi\rangle&=\langle u_n^+,g\varphi^+\rangle+\langle u_n^+,g\varphi^-\rangle=\langle u_n^+,g\varphi^+\rangle\\
    &=\langle u_n^+,g\varphi^+\rangle + \langle u_n^-,g\varphi^+\rangle =\langle u_n,g\varphi^+\rangle.
   \end{split}
    \]
    By the assumption $\lim_{n \to \infty} \sup_{g \in G} \langle u_n , g \varphi^+ \rangle = 0$, so taking the supremum and the limit, we get $u_n^+\stackrel{G}{\rightharpoonup} 0$. The proof for $u_n^-$ is analogous.
\end{proof}

Recall that $\cJ : X \rightarrow \R$ is a nonlinear functional of the form \eqref{E:introJ}. Assume in addition that $\cJ$ is $G$-invariant and that the following two conditions are satisfied
\begin{enumerate}
[label=(J\arabic{*}),ref=J\arabic{*}]\setcounter{enumi}{0}
    \item \label{J1} $\cI : X \rightarrow \R$ is of $\cC^1$-class with $\cI(0)=0$;
    \item \label{J2} $\cJ'$ is sequentially weak-to-weak* continuous.
\end{enumerate}
We introduce the following notation
$$
\critJ  := \{ u \in X \, : \, \cJ'(u) = 0 \}.
$$
As we are interested in a profile decomposition of Palais-Smale sequences, we note the following consequence of Theorem \ref{T:tintarev}.

\begin{Cor}\label{cor:1}
If $(u_n)$ is a bounded Palais-Smale sequence for $\cJ$, with $\cJ$ being $G$-invariant and satisfying (\ref{J1}) and (\ref{J2}), then, in the notion of Theorem \ref{T:tintarev},
\begin{itemize}
    \item[(i)] $w^k$ are critical points of $\cJ$;
    \item[(ii)] if, in addition, $\eta:=\inf_{u \in \critJ \setminus \{0\}} \|u\| >0$, then $K \leq \frac{\limsup_{n\to\infty} \|u_n\|^2}{\eta^2} < \infty.$
\end{itemize}
\end{Cor}

\begin{proof}
\begin{itemize}
\item[(i)] Note that, since $\cJ$ is $G$-invariant, (\ref{J1}) and (\ref{J2}) hold, and $G$ acts unitarly, we get
\begin{align*}
\left| \cJ'( w^k) (\varphi) \right| &= \lim_{n\to\infty} \left| \cJ' ( (g_n^k)^{-1} u_n) (\varphi) \right| =\lim_{n\to\infty} \left| \langle \nabla \cJ ( (g_n^k)^{-1} u_n) , \varphi \rangle \right| \\
&= \lim_{n\to\infty} \left| \langle  (g_n^k)^{-1} \nabla \cJ (  u_n) , \varphi \rangle \right| = \lim_{n\to\infty} \left| \langle   \nabla \cJ (  u_n) , g_n^k \varphi \rangle \right| \\
&\leq \lim_{n\to\infty} \| \nabla \cJ (u_n) \| \| g_n^k \varphi \| = \lim_{n\to\infty} \|\nabla\cJ (u_n) \| \| \varphi \| = 0
\end{align*}
for any $\varphi \in X$.
\item[(ii)] From Theorem \ref{T:tintarev}
\begin{align*}
\limsup_{n\to\infty} \| u_n\|^2 \geq \sum_{k=1}^K \|w^k\|^2 \geq \eta^2 K
\end{align*}
and therefore $K \leq \frac{\limsup_{n\to\infty} \|u_n\|^2}{\eta^2} < \infty$.
\end{itemize}
\end{proof}
\begin{Rem}
    Since $\cJ$ is $G$-invariant and satisfies (\ref{J1}), (\ref{J2}), the critical points $w^k$ can be chosen as any representatives of their orbits.
\end{Rem}

We introduce the following sequential, uniform $G$-weak-to-weak*-type continuity of $\cI'$:
\begin{equation}\label{GWC}\tag{GWC}
      \text{if $(v_n), (\varphi_n) \subset X$, $(v_n)$ is bounded and $\varphi_n \stackrel{G}{\rightharpoonup} 0$, then $\cI'(v_n)(\varphi_n) \to 0$.}
\end{equation}

Observe that if $\varphi_n \stackrel{G}{\rightharpoonup} 0$, then $\varphi_n \weakto 0$. Hence for strongly convergent sequence $v_n \to v \in X$, the property $\cI'(v_n)(\varphi_n) = \langle \nabla \cI (v_n), \varphi_n \rangle \to 0$ is clear. In (\ref{GWC}), under a stronger convergence on $\varphi_n$, we require that $\cI'(v_n)(\varphi_n) \to 0$ for every bounded sequence $(v_n)$.

\begin{Cor}\label{C:Tintarev_in_norm}
Suppose that $\cI'$ satisfies (\ref{GWC}). In the setting of Corollary \ref{cor:1}(ii) we have
$$
\left\| u_n  - \sum_{k=1}^K g_n^k w^k \right\| \to 0.
$$
\end{Cor}

\begin{proof}
Put $\xi_n := u_n  - \sum_{k=1}^K g_n^k w^k$. Since $(\xi_n)$ is bounded and $(u_n)$ is a Palais-Smale sequence, we get
$$
\left| \cJ' (u_n) ( \xi_n^\pm) \right| \leq \| \cJ'(u_n) \| \|  \xi_n^\pm \| \leq \|\cJ'(u_n) \| \| \xi_n\| \to 0.
$$
On the other hand,
\begin{align*}
\cJ'(u_n)( \xi_n^\pm) &= \pm \langle  u_n^\pm,  \xi_n^\pm \rangle - \cI'(u_n)(\xi_n^\pm) \\
&= \pm \| \xi_n^\pm\|^2 \pm \sum_{k=1}^K \langle g_n^k (w^k)^\pm, \xi_n^\pm \rangle - \cI'(u_n)(\xi_n^\pm) = (*).
\end{align*}
Since $g_n^k w^k$ are critical points of $\cJ$, we get
$$
0 = \cJ' (g_n^k w^k) ( \xi_n^\pm ) = \pm \langle g_n^k (w^k)^\pm, \xi_n^\pm \rangle - \cI'(g_n^k w^k) ( \xi_n^\pm ).
$$
Hence, since $(u_n)$ and $(g_n^k w^k)$ are bounded and $\xi_n^\pm \stackrel{G}{\rightharpoonup} 0$ (cf. Proposition \ref{P:D-convergence_of_projections})
\begin{align*}
(*) = \pm \| \xi_n^\pm\|^2 + \sum_{k=1}^K \cI'(g_n^k w^k) ( \xi_n^\pm ) - \cI'(u_n)(\xi_n^\pm) = \pm \| \xi_n^\pm \|^2 + o(1).
\end{align*}
Thus, $\|\xi_n^\pm\| \to 0$ and therefore $\|\xi_n\| \to 0$.
\end{proof}

In our work, we need an additional property of dislocation spaces, which we refer to as \emph{discreteness property}. 

For $\ell\in \N$ and a finite set $\cA\subset X$ we will denote
\[
[\cA,\ell]:=\left\{\sum_{i=1}^j g_i u_i\,:\, 1\leq j\leq \ell,\, g_i\in G,\, u_i\in\cA\right\}.
\]

\begin{Def}
    We say that a dislocation space $(X,G)$ has a discreteness property if  for any finite set $\cA\subset X$ and $\ell\in\bN$
    \[
    \inf\{\|u-u'\|\,:\, u,u'\in[\cA,\ell],\,u\neq u'\} >0.
    \]
\end{Def}
The discreteness property implies, among others things, that the orbit of each element $u\in X$ is discrete. This obviously excludes connected topological groups, like e.g. $\mathcal{SO}(N)$, from acting on dislocation spaces endowed with the discreteness property.

Below we give some examples of dislocation spaces $(X,G)$ having the discreteness property, with a description of $\stackrel{G}{\rightharpoonup}$ convergence. From now on $|\cdot|_r$ denotes the usual $L^r$-norm.

\begin{Ex}\label{ex:1}
Let $s \in (0,1]$, $N \geq 1$ and $H^s (\R^N)$ denote the usual Sobolev space. Then $(H^s (\R^N), \mathbb{Z}^N)$ is a dislocation space with the discreteness property, where the action of $\Z^N$ on $H^s (\R^N)$ is given by translations. Moreover, we have the following description of the topology:
\begin{itemize}
    \item[(i)] for $(z_n) \subset \Z^N$, $z_k \weakto 0$ if and only if $|z_k|\to\infty$;
    \item[(ii)] for a bounded sequence $(u_n) \subset H^s (\R^N)$, $u_n \stackrel{\Z^N}{\weakto} 0$ if and only if $|u_n|_p \to 0$ for any $p \in (2,2^*)$.
\end{itemize}
\end{Ex}

\begin{proof}
In the case $s=1$, the statement follows from \cite[Lemma 3.2]{Tintarev2007}, while for $s \in (0,1)$ the proof is similar. Again, in the case $s=1$, the discreteness property follows from \cite[Proposition 1.55]{CZR} and the proof in the case $s \in (0,1)$ is the same. Then the description of the topology follows from \cite[Lemma 3.1, Lemma 3.3]{Tintarev2007}.
\end{proof}

\begin{Ex}\label{ex:2}
Suppose that $N > K \geq 2$. Let $\cG(K) := \cO(K) \times \{ id_{N-K} \}$ be the group acting on $H^1 (\R^N)$. We will write $x = (y,z) \in \R^{K} \times \R^{N-K}$. Then we consider the space of $\cG(K)$-invariant functions $H^1_{\cG(K)} (\R^N)$. Then $(H^1_{\cG(K)} (\R^N), \Z^{N-K})$ is a dislocation space with the discreteness property, where the action of $\Z^{N-K}$ is given by translations with respect to the $z$-coordinate. Then, the topology is described as follows:
\begin{itemize}
    \item[(i)] for $(z_n) \subset \Z^{N-K}$, $z_k \weakto 0$ if and only if $|z_k|\to\infty$;
    \item[(ii)] for a bounded sequence $(u_n) \subset H^1_{\cG(K)} (\R^N)$, $u_n \stackrel{\Z^{N-K}}{\weakto} 0$ if and only if $|u_n|_p \to 0$ for any $p \in (2,2^*)$.
\end{itemize}
\end{Ex}

\begin{proof}
Recalling the concentration compactness principle \cite[Corollary 3.2, Remark 3.3]{MR4173560}, we can repeat the proof of \cite[Lemma 3.1, Lemma 3.2, Lemma 3.3]{Tintarev2007} to show that $(H^1_{\cG(K)} (\R^N), \Z^{N-K})$ is a dislocation space with the description of the topology given by (i) and (ii). To show the discreteness property, it is enough to repeat the proof of \cite[Proposition 4.3]{KS}.
\end{proof}

For the sake of completeness, we believe it could also be interesting to mention an example of a dislocation space without the discreteness property.

\begin{Ex}\label{ex:3}
Let $N\geq 3$, $\cD^{1,2} (\R^N) := \left\{ u \in L^{2^*}(\R^N) \ : \ \nabla u \in L^2 (\R^N) \right\}$. It is a Hilbert space with the norm $u \mapsto |\nabla u|_2$. Consider the action of $\R \times \R^N$ on $\cD^{1,2}(\R^N)$ given by
$$
(\R \times \R^N) \times \cD^{1,2}(\R^N) \ni ( (s, y), u) \mapsto 2^{\frac{N s}{2^*}} u(2^s (\cdot - y) ) \in \cD^{1,2} (\R^N).
$$
Then $(\cD^{1,2} (\R^N), \R \times \R^N)$ is a dislocation space, and we have the following description of the topology:
\begin{enumerate}
    \item[(i)] for $(s_n, y_n) \subset \R \times \R^N$, $(s_n, y_n) \weakto 0$ if and only if $|s_n| + |y_n| \to \infty$;
    \item[(ii)] for a bounded sequence $(u_n) \subset \cD^{1,2} (\R^N)$, $u_n \stackrel{\R \times \R^N}{\weakto} 0$ if and only if $|u_n|_{2^*} \to 0$.
\end{enumerate}
\end{Ex}

\begin{proof}
The statement follows directly from \cite[Lemma 5.1, Lemma 5.2, Lemma 5.3]{Tintarev2007}.
\end{proof}

Note that the dislocation space in Example \ref{ex:3} does not satisfy the discreteness property as the orbits of elements under the described action are not discrete. 

\subsection{\texorpdfstring{$\boldsymbol{\tau}$}{tau}-topology and admissible maps} 

For the proof of Theorem \ref{T:mainTheorem}, we rely on the $\tau$-topology, originally introduced in \cite{KS}. Let $(e_k)_{k=1}^\infty \subset X^-$ be a complete orthonormal sequence in the space $X^-$. Then we define a new norm $\triple{\cdot}$ in $X$ by
\[
	\triple{u} := \max \left\{ \| u^+ \|, \sum_{k=1}^\infty \frac{1}{2^{k}} \left| \langle u^-, e_k \rangle \right| \right\}.
\]
We denote by $\tau$ the topology on $X$ generated by $\triple{\cdot}$. We note that $\tau$ is weaker than the topology generated by the norm $\| \cdot \|$ and that the following inequalities hold
\[
	\|u^+\| \leq \triple{u} \leq \|u\|.
\]
We also recall that for bounded sequences $(u_n) \subset X$ the following equivalence holds true (see e.g. \cite[Remark 2.1(iii)]{KS})
\[
u_n \stackrel{\tau}{\to} u \quad \Longleftrightarrow \quad u_n^+ \to u^+ \mbox{ and } u_n^- \weakto u^-. 
\]
Based on \cite{KS}, we introduce the notion of admissible maps. Let $A \subset X$ be a closed subset. We say that a map $h : A \rightarrow X$ is \textit{admissible} if
\begin{itemize}
    \item it is \textit{$\tau$-continuous}, namely $h(u_n) \stackrel{\tau}{\to} h(u)$ if $u_n \stackrel{\tau}{\to} u$,
    \item the map $Id-h$, where $Id$ denotes the identity map, is \textit{$\tau$-locally finite-dimensional}, namely for every $u \in A$ there is a $\tau$-open neighborhood $U_u$ in $X$ such that $(Id-h)(U_u \cap A)$ is contained in a finite-dimensional subspace of $X$.
\end{itemize}

Let $W\subset X$ be a $\tau$-open set. We say that the map $h:W\to X$ is \textit{$\tau$-locally $\tau$-Lipschitzian} if for any $u\in W$ there is $\tau$-open neighborhood $U_u\subset W$ such that $\triple{h(u')-h(u'')}\leq L_u\triple{u'-u''}$ for all $u',u''\in U_u$ and some $L_u > 0$.

\begin{Rem}\label{R:admissibleFlow}
    Let $V:W\to X$ be $\tau$-locally $\tau$-Lipschitzian, locally Lipschitzian and $\tau$-locally finite-dimensional. Let $A\subset W$ be closed and assume that the flow generated by $V_{|A}$ exists on $[0,1]$; denote it by $\eta:A\times [0,1]\to X$. Then the map $\eta(\cdot,1):A\to X$ is an admissible map, see \cite[Proposition 2.2]{KS}. Note that admissible maps are continuous due to the equivalence of the norms $\triple{\cdot}$ and $\|\cdot\|$ on finite-dimensional subspaces of $X$.
\end{Rem}

Put $\Sigma=\left\{A \subset X \ : \  A=-A=\overline{A} \right\}$. For the set $A \in \Sigma$, assuming that $\cJ:X\to\bR$ is an even functional, we define
\begin{equation}\label{DefH}
	\cH(A,\cJ)=\left\{h:A \to X : \begin{array}{l} h(A) \mbox{ is closed, }\\ h \mbox{ is an odd, \textit{admissible} map, } \\ \mbox{for any } u \in A \mbox{ there holds } \cJ(h(u)) \leq \cJ(u)  \end{array} \right\}.
\end{equation}
If the functional $\cJ$ is known from the context, we will write $\cH(A)$ instead of $\cH(A,\cJ)$.
Now we recall the definition of the Krasnoselskii genus (see e.g. \cite[Definition II.5.1]{Struwe}). 
\begin{Def}\label{D:Kgenus} 
For $A \in \Sigma$ we define 
\[
	\gamma(A):= \inf\left\{k \in \N \ : \ \mbox{there exists an odd, continuous } \vp: A \to \R^k \setminus \{0\}\right\}.
\]
We set $\gamma(\emptyset)=0$.
\end{Def}
\noindent Note that if $A \in \Sigma$ and $0\in A$ then the set in the definition of $\gamma(A)$ is empty and $\gamma(A)=\infty$. One can find more properties of the Krasnoselskii genus in \cite{Rab86, Struwe}. In the next section we will construct some extensions of this tool.

\section{Abstract critical point theory}
\label{sect:abstr}

Let $(X,\langle \cdot, \cdot \rangle)$ be a real, separable Hilbert space with a given action of a group $G$ such that $(X,G)$ is a dislocation space with the discreteness property. Let $\cJ:X\to\bR$ be a $G$-invariant functional of the class $\cC^1$. This symmetry implies that the set of critical points $\crit(\cJ)$ consists of $G$-orbits $\cO(u):=\{gu\,:\,g\in G\}$. We say that critical points $u_1, u_2\in\critJ$ are \emph{geometrically distinct} if $\cO(u_1)\neq \cO(u_2)$.

We are going to formulate and prove an abstract theorem about the existence of multiple orbits of critical points of $\cJ$. Our idea is based on Theorem 4.1 from \cite{KS} and the flow of our proof will be similar, with two main differences: we state and prove the abstract theorem under the assumptions given on the functional, and we do not assume that the functional is $\tau$-upper-semicontinuous which was an important property of the reasoning given in \cite{KS}, but does not hold in our applications.

We consider the following assumptions
\begin{enumerate}
[label=(A\arabic{*}),ref=A\arabic{*}]\setcounter{enumi}{0}
\item \label{A1} $(X,G)$ is a dislocation space with the discreteness property and $X=X^+\oplus X^-$, where $X^+$ is infinite dimensional, $X^+$ and $X^-$ are orthogonal and $G$-invariant;
\item \label{A2} $\cJ:X\to\bR$ is given by $\cJ(u)=\frac 12\|u^+\|^2-\frac 12 \|u^-\|^2 - \cI(u)$, where $\cI : X \to \R$ is of $\cC^1$-class, $\cI'$ is sequentially weak-to-weak* continuous, satisfying (\ref{GWC})-property, and $\cI(0)=0$;
\item \label{A3} $\cJ$ is even and $G$-invariant;
\item \label{A4} $\critJ \setminus \{0\} \neq \emptyset$;
\item \label{A5} there exist $r>0$ and $r_0>0$ such that
\[
b:=\inf\limits_{\cS_{r}^+} \cJ > \sup\{ \cJ(u) \ : \ u \in X, \ \|u^+\|<r_0 \}=:a,
\]
where $\cS^{+}_r:= \{x\in X^{+}\ :\ \|x\|=r\}$.
\item \label{A6} if $\|u_n^-\|\to \infty$ and $\|u_n^+\|$ is bounded then $\cJ(u_n)\to -\infty$;
\item \label{AX} $\cJ$ is bounded from above  on bounded sets;
\item \label{A7} $\frac{\cI(u_n^+)}{\|u_n^+\|^2}\to\infty$, if $\|u_n^+\|\to\infty$ and $(u_n^+)$ is contained in a finite-dimensional subspace of $X^+$;
\item \label{A8} for every $\beta \in \R$ there is a constant $M_\beta > 0$ such that for every sequence $(u_n)$ satisfying
$$
\limsup_{n\to\infty} \cJ(u_n) \leq \beta, \quad \cJ'(u_n) \to 0
$$
there holds $\limsup_{n\to\infty} \|u_n\| \leq M_\beta$.
\end{enumerate}

\begin{Rem}\label{R:A7}
Note that (\ref{A7}) implies that $\cJ(u_n^+) \to -\infty$, if $\|u_n^+\|\to\infty$ and $(u_n^+)$ is contained in a finite-dimensional subspace of $X^+$. Indeed, suppose by contradiction that $\cJ(u_n^+) \geq C$ for some $C \in \R$. Then
$$
o(1) = \frac{C}{\|u_n^+\|^2} \leq \frac{\cJ(u_n^+)}{\|u_n^+\|^2} = \frac12 - \frac{\cI(u_n^+)}{\|u_n^+\|^2} \to -\infty,
$$
which is a contradiction.
\end{Rem}

We are now ready to state our main theorem, which reads as follows.

\begin{Th}\label{T:mainTheorem}
Under the assumptions (\ref{A1})--(\ref{A8}) the functional $\cJ$ has infinitely many geometrically distinct critical points.
\end{Th}

Before starting the proof, some comments are in order.

\begin{Rem}
In (\ref{A4}), we assume that the functional $\cJ$ already has at least one nontrivial critical point. Hence, the above theorem is purely of multiplicity-type. However, if we replace (\ref{A4}) by assumptions that provide the existence of a Palais-Smale-type sequence, we are still able to obtain the result. For this version of the theorem, see Appendix \ref{appendix}.
\end{Rem}

\begin{Rem}
Note that, under our assumptions, the functional is not necessarily $\tau$-upper-semicontinuous, so we cannot directly apply the results from \cite{KS}. Similarly, if instead of $\tau$ we consider $\widetilde{\tau}$ being a product of the strong topology in $X^+$ and weak topology in $X^-$, $\cJ$ is still not necessarily $\widetilde{\tau}$-upper-semicontinuous, so the results in \cite{BartschDing} do not apply. Such conditions in \cite{BartschDing, KS} are implied by the positivity of the nonlinear term, and have been replaced here by the weak-to-weak* continuity and (\ref{GWC}) property of $\cI'$, see (\ref{A2}). Instead of the classical linking geometry, we require (\ref{A5}). Assumptions (\ref{A6}) and (\ref{A7}) are quite standard, in the spirit of assumptions (I4), (I7) in \cite{MSS, BiegMed2}.
\end{Rem}

Hereafter, we work under assumptions (\ref{A1})--(\ref{A8}). First, we formulate a couple of technical lemmas and then prove the main theorem. From now on, we denote by $P$ and $Q$ the orthogonal projections onto $X^-$ and $X^+$, respectively. Moreover, for a given norm $\| \cdot \|$ on $X$, $\|u - A \|$ denotes the distance of $u \in X$ from the subset $A \subset X$.

We will proceed by contradiction, i.e. assuming that 
\begin{equation}\label{F-finite}
\critJ \mbox{ consists of finitely many orbits}
\end{equation}
and denoting by $\cF$ a finite set of arbitrarily chosen representatives of these orbits. Since $\cJ$ is odd we may replace $\cF$ by $\cF \cup (-\cF)$, and therefore we can assume that $\cF=-\cF$. We will use the notation $\cF^+=Q\cF=\{u^+\,:\, u\in \cF\}$.

We start our reasoning by defining the number $\beta$ such that the superlevel set $\cJ_{\beta}$ is away from the finite set $\cF$ in the norm $\triple{\cdot}$. Recall that the level sets of $\cJ$ are $G$-invariant therefore we get also that $\cJ_{\beta}$ is $\tau$-away from $\critJ$.
\begin{Lem}\label{L:awayFromCrit}
    There is a number $\beta_0$ such that for $u\in \cJ_{\beta_0}$ we have $\triple{u-\cF}\geq 1$.
\end{Lem}
\begin{proof}
    Let $u\in\cF$ and consider a set $B=\{v\in X\,:\,\triple{v-u}< 1\}$. Since $\|v^+-u^+\|\leq \triple{v-u}< 1$  the norm $\|v^+\|$ is bounded for $v\in B$. By assumption (\ref{A6}), there is $R>0$ such that if $\|v^-\|>R$ and $v \in B$, then $\cJ(v)<0$. Put $s_{u}=\sup \{\cJ(v)\,:\, v \in B, \,\|v^-\|\leq R\}$. Thanks to (\ref{AX}), $s_u$ is finite and we can finally define
    \[
        \beta_0=\max_{u\in\cF} s_u.
    \]
    
\end{proof}

Note that $b>a\geq \cJ(0)=0$, and fix $\alpha>1$ and $\beta>\max\{b,\beta_0\}>0$ such that
\[
\beta>\max\limits_{\cF} \cJ\geq\min\limits_{\cF} \cJ>-\alpha+1.
\]

Since $(X,G)$ is a dislocation space, $\cI$ satisfies the condition (\ref{GWC}), and (\ref{A8}) holds, the assumptions of Corollaries \ref{cor:1} and \ref{C:Tintarev_in_norm} are satisfied,  giving rise to the following result.

\begin{Lem}\label{L:PS-approximation}
    Under the assumptions of Theorem \ref{T:mainTheorem}, for any $c\in\bR$ there is $\ell_c \geq \frac{M_c^2}{\eta}$ such that for every $(PS)_c$ sequence there holds
    \[
        0\leq \triple{u_n-[\cF,\ell_c]} \leq \|u_n-[\cF,\ell_c]\| \to 0,
    \]
    where the lower bound $\frac{M_c^2}{\eta}$ comes from Corollary \ref{cor:1} and (\ref{A8}).
\end{Lem}
\begin{proof}
From assumption (\ref{A8}), $(u_n)$ is bounded. Applying Theorem \ref{T:tintarev} and Corollaries \ref{cor:1}, and \ref{C:Tintarev_in_norm} (in particular, from \eqref{F-finite}, $K < \infty$), we get that
$$
\left\| u_n  - \sum_{k=1}^K g_n^k w^k \right\| \to 0
$$
for some $g_n^k \in G$ and critical points $w^k \in \cF$. Clearly $\sum_{k=1}^K g_n^k w^k \in [\cF, \ell]$ for $\ell \geq K$.
\end{proof}

We are going to prove the main part of our argument, i.e. the suitable deformation lemma.
\begin{Lem}\label{L:deformationLemma}
    Let $\xi > \beta+2$. There exist $1>\varepsilon>0$, a symmetric $\tau$-open set $\cN\subset X$ with $\gamma(\overline{\cN})=1$, and a map $h\in \cH(\cJ^{\xi})$ such that
    \begin{itemize}
        \item[(i)] for any $d\in[b,\xi-1]$, $h(\cJ^{d+\varepsilon}\setminus\cN)\subset \cJ^{d-\varepsilon}$;
        \item[(ii)] moreover, if $d\geq \beta+1$, then $h(\cJ^{d+\varepsilon})\subset \cJ^{d-\varepsilon}$.
    \end{itemize}   
    \end{Lem}

\begin{proof} The proof consists of three steps.

\textbf{Step I. Construction of $\cN$.} 
     The set $\cF$ is assumed to be finite, so we can fix $\ell:= \ell_{\xi+2}$ from Lemma \ref{L:PS-approximation}. Since $(X,G)$ has a~discreteness property and $\cF^+$ is also finite there exists $\mu>0$ such that $\mu<\min\{r_0,1\}$, with $r_0$ given in (\ref{A5}), and
    \begin{equation}\label{E:infimum}
        \inf\{\|u-u'\|\,:\, u,u'\in[\cF^+,\ell],\,u\neq u'\}\geq \mu.
    \end{equation}
    We will define $\cN$ as a neighborhood of $[\cF,\ell]\setminus\{0\}$. For any $z\in[\cF^+,\ell]$ we take an open ball in $X^+$ centered at $z$ with small radius $\frac{\mu}4$. From \eqref{E:infimum} these balls are disjoint. Put
    \[
	\cN:= \bigcup_{z \in \left[\cF^+,\ell\right] \setminus \{0\}} \left(X^- \oplus B_{X^+}\left(z,\dfrac{\mu}{4}\right)\right)=X^-\oplus\bigcup_{z \in \left[\cF^+,\ell\right] \setminus \{0\}} B_{X^+}\left(z,\dfrac{\mu}{4}\right).
\]
Since $\cF^+$ and $\left[\cF^+,\ell\right]$ are symmetric, it is clear that $\cN$ is a symmetric set. On $X^+$ the $\tau$-topology is the strong one, so $\cN$ is a $\tau$-open set. Since $\overline{\cN}$ can be contracted to a finite collection of points, we have $\gamma(\overline{\cN})=1$ (see Remark 7.3 in \cite{Rab86}). Note that $0\notin \cN$ and therefore $\gamma(\overline{\cN})$ is well defined.

Consider the set
\begin{equation}\label{E:N0-def}
    \cN_8:=X^-\oplus\bigcup_{z \in \left[\cF^+,\ell\right]} B_{X^+}\left(z,\dfrac{\mu}{8}\right)\subset\cN \cup \left( X^-\oplus B_{X^+} \left(0,\dfrac{\mu}{8}\right)\right).
\end{equation}
In view of Lemma \ref{L:PS-approximation}, there is $\delta>0$ such that $\|\nabla \cJ(u)\|\geq \delta$ provided $u\in \cJ_{\lowlevel}^{\xi+2}\setminus\cN_8$. 

\textbf{Step II. Vector field and flow.} Now we construct the vector field that we use to define the map $g$ as a generated flow. Firstly, let $w:\lJ{\lowlevel}{\xi+2}\setminus \critJ\rightarrow X$ be given by
\[
w(u)=\frac{2\nabla \cJ(u)}{\|\nabla \cJ(u)\|^2}.
\]
Since the map $\nabla \cJ$ is weak-to-weak* sequentially continuous, the function $v\mapsto \langle\nabla\cJ(v),w(u)\rangle$ is $\tau$-continuous on $\cJ_{-\alpha}^{\xi+2}$ for any $u \in \lJ{\lowlevel}{\xi+2}\setminus \critJ$. Indeed, let $v_n\stackrel{\tau}{\to} v$. Then $(Qv_n)$ is bounded. If $(Pv_n)$ is unbounded, thanks to \eqref{A6}, $\cJ(v_n) \to -\infty$ along a subsequence, which is a contradiction. Hence, $(Pv_n)$ is bounded as well. Then $v_n\weakto v$ and $\langle \nabla \cJ(v_n),w(u)\rangle\to\langle \nabla \cJ(v),w(u)\rangle.$ Moreover,
\[
\langle \nabla\cJ(u),w(u)\rangle =2,
\]
thus there exists a $\tau$-open neighborhood $U_u\subset X$ of $u\in \lJ{\lowlevel}{\xi+2}\setminus \critJ $, such that 
\begin{equation}\label{E:U_sets-properties}
    \triple{v-u}<\frac{1}{16}\mu\quad\text{and}\quad\langle \nabla\cJ(v),w(u)\rangle >1
\end{equation}
for $v\in  U_u$. 

This way we have constructed the $\tau$-open covering $\{U_u\}$ of $\lJ{\lowlevel}{\xi+2}\setminus\critJ$. Let $\cU := \bigcup U_u$. Then $\{U_u\}$ is an open covering of a metric space $(\cU, \tau)$. Hence, there is a locally finite refinement $\{M_j\}_{j\in J}$ and denote by $\{\lambda_j\}_{j\in J}$ the subordinated $\tau$-Lipschitzian partition of unity, $\lambda_j : \cU \rightarrow [0,1]$.
For any $j\in J$ we arbitrarily choose $u_j\in \lJ{\lowlevel}{\xi+2}\setminus\critJ$ such that $M_j\subset U_{u_j}$.

Note that $\bigcup_{j\in J} M_j$ does not need to be a symmetric set, therefore we take a $\tau$-open and symmetric set $M$ such that $\lJ{\lowlevel}{\xi+2}\setminus\critJ\subset M\subset \bigcup_{j\in J} M_j = \cU$. On this set, we define a~map $V:M\to X$ as
\[
V(u):=\frac 12 \left[ \widetilde{V}(u)-\widetilde{V}(-u)\right],\quad\text{where }\widetilde{V}(u)=\sum_{j\in J} \lambda_j(u)w(u_j).
\]
The map $V$ has the following properties.
\begin{enumerate}
    \item $V$ is odd.
    \item Since $\nabla\cJ$ is odd, we have 
    \[
    \langle \nabla\cJ(u),V(u)\rangle=\frac 12\left[ \langle \nabla\cJ(u),\widetilde{V}(u)\rangle + \langle \nabla\cJ(-u),\widetilde{V}(-u)\rangle \right].
    \]
    Moreover, for $u\in M$  we can compute
    \[
    \langle \nabla \cJ(u),\widetilde{V}(u)\rangle=\sum_{j\,:\,u\in M_j\subset U_{u_j}} \lambda_j(u)\langle \nabla\cJ(u),w(u_j)\rangle>\sum_{j\,:\,u\in M_j\subset U_{u_j}} \lambda_j(u)=1.
    \]
    Therefore $\langle \nabla\cJ(u),V(u)\rangle>1$ for $u\in\lJ{\lowlevel}{\xi+2}\setminus\critJ\subset M$.
    \item For every $u\in M$ there is $\tau$-open neighborhood $L_u\subset M$ such that the set $\{j\in J\,:\, M_j\cap L_u \neq \emptyset\}$ is finite. Hence $V(L_u)$ is contained in a finite-dimensional subspace of $X$.
    \item Since the functions $\lambda_j$, $j\in J$, are $\tau$-Lipschitzian and we have $\triple{u-v}\leq \| u-v\|$, we see that $V$ is $\tau$-locally $\tau$-Lipschitzian and locally Lipschitzian.
\end{enumerate}

We define two even, $\tau$-locally $\tau$-Lipschitzian (and therefore locally Lipschitzian also) functions $\theta, \psi:X\to [0,1]$ that vanish outside $\lJ{\lowlevel}{\xi+2}$ and close to $\critJ$, respectively. More precisely: 
\begin{itemize}
    \item $\theta(u)=1$ on $\lJ{\lowlevel+1}{\xi+1}$ and $\theta(u)=0$ on $\cJ_{\xi+2}\cup \cJ^{-\alpha}$;
    \item $\psi(u)=0$ if $\triple{u-\critJ}\leq \frac 1{10}\mu$ and $\psi(u)=1$ if $\triple{u-\critJ}\geq \frac 1{8}\mu$.
\end{itemize} 
Define $\Psi : X \rightarrow X$ by $\Psi:=\theta\psi V$ and consider the Cauchy problem
\begin{equation}\label{E:CP1}
\frac{\dd\eta}{\dd t}=-\Psi(\eta),\quad \eta(u,0) = u \in \cJ^{\xi+1}.
\end{equation}
The map $\Psi$ is locally Lipschitzian, hence the problem \eqref{E:CP1} has a unique solution defined on the open interval $(\omega_-(u),\omega_+(u))$ containing $0$. Moreover the map $\Psi$ is odd, therefore $\omega_{\pm}(u)=\omega_{\pm}(-u)$ and $\eta(-u,t)=-\eta(u,t)$ for $u\in \cJ^{\xi+1}$ and $t\in (\omega_-(u),\omega_+(u))$. 

\textbf{Claim.} The solution of the Cauchy problem \eqref{E:CP1} is global for any $u\in \cJ^{\xi+1}$, i.e. $\omega_{\pm}(u)=\pm\infty$.

\textbf{Proof of Claim.} 
We will prove that $\omega_+(u)=\infty$, the proof of the second case follows in a similar way. We argue by contradiction. Suppose that $\omega_+(\tu)$ is finite for some $\tu\in\J^{\xi+1}$.

By standard argument, the map $\Psi$ cannot be bounded along trajectories. Therefore, there exists an increasing sequence $(t_m)$ such that $t_m\to \omega_{+}(\tu)$, $\Psi(\eta(\tu,t_m))\neq 0$ and $\| V(\eta(\tu,t_m))\|\to\infty$. Put $z_m=\eta(\tu,t_m)$. Since the sum in the definition of the map $V$ is locally finite, for all $m$ we find an index $j(m)\in J$ such that $\|w(u_{j(m)})\|\to\infty$ and $\lambda_{j(m)}(\eta(\tu,t_m))\neq 0$  (taking a subsequence or switching from $u$ to $-u$ if necessary). It means that 
\[
\frac{2}{\|\nabla \cJ(u_{j(m)})\|}=\|w(u_{j(m)})\|\to\infty \Leftrightarrow \|\nabla \cJ(u_{j(m)})\|\to 0 \Leftrightarrow \cJ'(u_{j(m)})\to 0.
\]

Since $(u_{j(m)})\subset \lJ{\lowlevel}{\xi+2}\setminus\critJ$ we have $(\cJ(u_{j(m)}))\subset [-\alpha,\xi+2]$ and therefore $(u_{j(m)})$ is a $(PS)$-sequence. By Lemma \ref{L:PS-approximation} we get $\triple{u_{j(m)}-[\cF,\ell]}\to 0$ and as a consequence $\triple{Qu_{j(m)}-[\cF^+,\ell]}\to 0$.

Note that $\lambda_{j(m)}(\eta(\tu,t_m))\neq 0$ means $\eta(\tu,t_m)\in M_{j(m)}\subset U_{u_{j(m)}}$ and in the view of \eqref{E:U_sets-properties} we have $\triple{\eta(\tu,t_m)-u_{j(m)}}<\frac{\mu}{16}$.

Now, two cases can occur, and both of which must be ruled out.

\emph{Case 1.} Suppose there is $z\in[\cF^+,\ell]$ such that $u_{j(m)}\in X^-\oplus B_{X^+}\left(z,\frac{\mu}{16}\right)$ for almost all $m$. Then $Qu_{j(m)}\to z$. Moreover, since $\|Qu_{j(m)}\|$ is bounded and $\cJ(u_{j(m)})$ is bounded from below, due to the assumption (\ref{A6}) we obtain that the sequence $(Pu_{j(m)})_m$ is bounded,
and therefore weakly convergent to $y\in X^-$. Finally, $u_{j(m)}\stackrel{\tau}{\weakto} y+z \in X$. The map $\cJ'$ is sequentially weak-to-weak* continuous, therefore $\cJ'(y+z)=\lim_{m\to\infty} \cJ'(u_{j(m)})=0 $. Hence $y+z\in \critJ$ and
\begin{align*}
\triple{\eta(\tu,t_m)-\critJ} &\leq \triple{\eta(\tu,t_m)-u_{j(m)}}+\triple{u_{j(m)}-\critJ} \\
&< \frac{\mu}{16}+\triple{u_{j(m)}-\critJ}<\frac{\mu}{10},
\end{align*}
for almost all $m$. Therefore, by the definition of the map $\psi$ we have $\Psi(\eta(\tu,t_m))=0$, a~contradiction.

\emph{Case 2.} There are infinitely many $z\in[\cF^+,\ell]$ such that $u_{j(m)}$ enters the sets $X^-\oplus B_{X^+}\left(z,\frac{\mu}{16}\right)$. Then, since $\triple{\eta(\tu,t_m)-u_{j(m)}}<\frac{\mu}{16}$,  the sequence $\eta(\tu,t_m)$ enters the sets of the form $X^-\oplus B_{X^+}\left(z,\frac{\mu}{8}\right)$ for infinitely many $z\in [\cF^+,\ell]$. It means that the flow $\eta(\tu, t)$ switches among disjoint sets $X^-\oplus B_{X^+}\left(z,\frac{\mu}{8}\right)\subset \cN_8$ for $t$ arbitrarily close to $\omega_+ (\tu)$.

Put $\cN_{16}:=X^-\oplus\bigcup_{z \in \left[\cF^+,l\right] } B_{X^+}\left(z,\dfrac{\mu}{16}\right)$. Let $t_1<t_2<\omega_+(\tu)$ and $z_1\neq z_2$ be such that 
\[
\eta(\tu,t_1)\in X^-\oplus \overline{B_{X^+}\left(z_1,\frac{\mu}8\right)},\quad \eta(\tu,t_2)\in X^-\oplus \overline{B_{X^+}\left(z_2,\frac{\mu}8\right)}
\]
and $\eta(\tu,t)\notin \cN_8$ for $t\in(t_1,t_2)$. Since $z_1,z_2\in [\cF^+,l]$ due to \eqref{E:infimum} we have 
\[
\|\eta(\tu,t_1)-\eta(\tu,t_2)\|\geq \|z_1-z_2\|-\left(\|z_1-\eta(\tu,t_1)\|+\|z_2-\eta(\tu,t_2)\|\right)\geq \mu -2\cdot\frac{\mu}8=\frac 34 \mu.
\]
On the other hand for $t\in(t_1,t_2)$ there is
\[
\|\widetilde{V}(\eta(\tu,t))\|=\left\|\sum_{j\in J_0} \lambda_j(\eta(\tu,t)) w(u_j)\right\|\leq \sup_{j\in J_0} \|w(u_j)\|,
\]
where the set $J_0$ is finite and, by construction, if $j\in J_0$ then $\triple{u_j-\eta(\tu,t)}<\frac{\mu}{16}$. Therefore, since  $\eta(\tu,t)\notin \cN_8$, $u_j\notin \cN_{16}$. Note, as for $\cN_8$, that there exist $\delta_1>0$ such that $\|\nabla\cJ(u)\|\geq \delta_1$ for $u\in \lJ{-2\alpha}{\xi+2}\setminus \cN_{16}$. Hence
\begin{equation}\label{E:temp}
    \|\widetilde{V}(\eta(\tu,t))\|\leq \sup_{j\in J_0} \|w(u_j)\|= \sup_{j\in J_0} \frac{2}{\|\nabla \cJ(u_j)\|}\leq \frac{2}{\delta_1}.
\end{equation}
Finally, we get
\[
\frac{3}{4}\mu\leq  \|\eta(\tu,t_1)-\eta(\tu,t_2)\|\leq \int_{t_1}^{t_2}\|V(\eta(\tu,s))\|\, ds\leq \frac{2}{\delta_1}(t_2-t_1)
\]
leading to a contradiction, since $t_1$ and $t_2$ may be chosen arbitrarily close to $\omega_+(\tu)$. This concludes the study of the second case and the proof of the claim.

\textbf{Step III. Definition and properties of $h$.}

Since the map $\Psi$ is $\tau$-locally $\tau$-Lipschitzian and $\tau$-locally finite-dimensional, the map $h:\cJ^{\xi}\to X$ given by $h(u):=\eta(u,1)$ is an admissible odd map, see Remark \ref{R:admissibleFlow}. To prove that $\cJ(h(u))\leq \cJ(u)$ we compute
\[
\frac{d}{dt}\cJ(\eta(u,t))=\langle \nabla\cJ(\eta(u,t)),-\Psi(\eta(u,t))\rangle=-[\theta\psi](\eta(u,t))\langle \nabla\cJ(\eta(u,t)),V(\eta(u,t))\rangle \leq 0,
\]
due to the property $(2)$ of the map $V$. It means that $\cJ$ is non-increasing along trajectories and, in particular, $\cJ(h(u))=\cJ(\eta(u,1))\leq \cJ(\eta(u,0)=\cJ(u)$. Note that $h(\cJ^{\xi})\subset \cJ^{\xi}$. To prove that $h(\cJ^{\xi})$ is closed in $X$ consider a sequence $v_m=h(z_m)=\eta(z_m,1)$, where $z_m\in \cJ^{\xi}$ and $v_m\to v$. Applying the continuous map $\eta(\cdot,-1)$ we obtain
\[
z_m=\eta(v_m,-1)\to \eta(v,-1)=:z.
\]
Since the set $\cJ^{\xi}$ is closed, $z\in \cJ^{\xi}$ and finally $v=\eta(z,1)=h(z)$, i.e. $v\in h(\cJ^{\xi})$ and the set $h(\cJ^{\xi})$ is closed. To summarize, $h\in\cH(\cJ^{\xi})$.

Now we will show that $h$ has the properties \textit{(i)} and \textit{(ii)} required in the thesis of Lemma \ref{L:deformationLemma}.
Put 
\[
\varepsilon=\frac 12 \min \left\{\frac 14, b-a, \frac{1}{32} \delta\mu\right\},
\]
where $b$ and $a$ are defined in the assumption (\ref{A5}), $\mu$ comes from \eqref{E:infimum} and $\delta$ is defined below \eqref{E:N0-def}.

\textbf{Property \textit{(i)}} Since $b-\varepsilon> a$, due to assumption (\ref{A5}) and the inequality $\mu<r_0$ we have
\[
\overline{\cN_8}\setminus\cN=\left\{u\,:\, \|u^+\|\leq\frac{\mu}8\right\}\subset \left\{u\,:\, \|u^+\|<r_0\right\}\subset \cJ^{a}\subset \cJ^{b-\varepsilon}.
\]

Take any $d\in[b,\xi-1]$ and suppose there is $\tu\in \cJ^{d+\varepsilon}_{d-\varepsilon}\setminus \cN\subset\cJ^{d+\varepsilon}_{d-\varepsilon}\setminus \overline{\cN_8}$ such that $\eta(\tu,1)\subset\cJ^{d+\varepsilon}_{d-\varepsilon}$. Then $\eta(\tu,t)\in\cJ^{d+\varepsilon}_{d-\varepsilon}\subset \cJ^{\xi+1}_{\lowlevel+1}$ for $t\in[0,1]$, since $d-\varepsilon \geq b-\varepsilon>0$, and therefore $\theta(\eta(\tu,t))=1$.

\noindent If $\eta(\tu,t)\in\cJ^{d+\varepsilon}_{d-\varepsilon}\setminus \cN_8$ for all $t\in[0,1]$, then $\psi(\eta(\tu,t))=1$ and
\[
\begin{split}
2\varepsilon\geq \cJ(\eta(\tu,0))-\cJ(\eta(\tu,1))&=\int_0^1 \langle \nabla \cJ(\eta(\tu,s)),\Psi(\eta(\tu,s))\rangle\, ds=\\
&= \int_0^1 \langle \nabla \cJ(\eta(\tu,s)),V(\eta(\tu,s))\rangle\, ds \geq 1,
\end{split}
\]
a contradiction.

Recall that $\frac{\mu}{4}<r_0$, and by the assumption (\ref{A5}), for $u\in X^-\oplus \overline{B_{X^+}\left(0,\frac{\mu}4\right)}$ we have $J(u)\leq a <b-\varepsilon\leq d-\varepsilon$. Therefore, since $\tu\in \cJ^{d+\varepsilon}_{d-\varepsilon}\setminus \cN$, there holds 
 $\|Q\tu-[\cF^+,\ell]\|\geq \frac{\mu}4$. Moreover $\tu\notin \overline{\cN_8}$. If there is $t_0\in(0,1]$ such that $\eta(\tu,t_0)\in \overline{\cN_8}$ (i.e. $\eta(\tu,t_0)\in X^-\oplus \overline{B_{X^+}\left(z,\frac{\mu}8\right)}$ for some $z\in[\cF^+,\ell]$) and $\eta(\tu,s)\notin \cN_8$ for $s\in[0,t_0)$, then $\Psi(\eta(\tu,s))=V(\eta(\tu,s))$ and we have
\[
    \begin{split}
        \frac{\mu}8&=\frac{\mu}{4}-\frac{\mu}8\leq \|Q\tu-z\|- \|Q\eta(\tu,t_0)-z\|\leq \|Q\tu-Q\eta(\tu,t_0)\| \\
        &\leq \|\tu-\eta(\tu,t_0)\|\leq \int_0^{t_0} \|V(\eta(\tu,s))\|\,ds\leq \frac{2}{\delta}t_0,
    \end{split}
\]
where the last inequality comes exactly the same way as in \eqref{E:temp}. It gives us the condition $t_0\geq \frac{\delta\mu}{16}>2\varepsilon$. On the other hand
\[
    \begin{split}
        2\varepsilon\geq \cJ(\eta(\tu,0)) - \cJ(\eta(\tu,t_0)) &=\int_0^{t_0} \langle \nabla \cJ(\eta(\tu,s)),\Psi(\eta(\tu,s))\rangle\, ds=\\
        &= \int_0^{t_0} \langle \nabla \cJ(\eta(\tu,s)),V(\eta(\tu,s))\rangle\, ds \geq t_0,
    \end{split}
\]
a contradiction.

\textbf{Property \textit{(ii)}} If $\xi-1\geq d\geq\beta+1$, we are far away from the set $\critJ$. More precisely, since $\beta\geq\beta_0$, by Lemma \ref{L:awayFromCrit} we have $\triple{u-\critJ}\geq 1>\frac{\mu}8$ for $u\in \cJ_{d-\varepsilon}^{d+\varepsilon}$. 

Suppose there is $\tu\in \cJ^{d+\varepsilon}_{d-\varepsilon}$ such that $h(\tu)=\eta(\tu,1)\subset\cJ^{d+\varepsilon}_{d-\varepsilon}$. Then $\eta(\tu,t)\in\cJ^{d+\varepsilon}_{d-\varepsilon}\subset \cJ^{\xi+1}_{\lowlevel}$ for $t\in[0,1]$ and therefore $\theta(\eta(\tu,t))=1$ and $\psi(\eta(\tu,t))=1$. Finally,
\[
    \begin{split}
        2\varepsilon\geq \cJ(\eta(\tu,0))-\cJ(\eta(\tu,1))&=\int_0^1 \langle \nabla \cJ(\eta(\tu,s)),\Psi(\eta(\tu,s))\rangle\, ds=\\
        &= \int_0^1 \langle \nabla \cJ(\eta(\tu,s)),V(\eta(\tu,s))\rangle\, ds \geq 1,
    \end{split}
\]
a contradiction.

This completes the proof of Lemma \ref{L:deformationLemma}.
   
\end{proof}

\subsection*{Pseudo-indices}
In order to prove our multiplicity result, we need to use two generalizations of the Krasnoselskii genus (see Definition \ref{D:Kgenus}), namely, some variants of the so-called Benci's pseudoindex (or $\Z^2$-index, see \cite{Benci82}). These definitions depend on the choice of the functional, so hereafter we fix $\cJ$ satisfying assumptions (\ref{A1})--(\ref{A8}) of Theorem \ref{T:mainTheorem}.
Recall the notation $\Sigma=\left\{A \subset X \ : \  A=-A=\overline{A} \right\}$.
\begin{Def}
For $A \in \Sigma$ we define
\[
	\gamma^*(A):=\inf_{h \in \cH(A)} \gamma\left(h(A) \cap \cS_{r}^+\right),
\]
where $r$ comes from assumption (\ref{A5}).
\end{Def}
Recall that the family $\cH(A,\cJ)$, defined in \eqref{DefH}, depends on the functional $\cJ$, but since it is fixed, we use the short notation $\cH(A)$. However, we should remember that the definition of pseudoindex $\gamma^*$ depends on $\cJ$. Note that since $h(A)$ is closed for $A \in \Sigma$, then $h(A) \cap \cS_{r}^+ \in \Sigma$ and the index is well-defined.

We recall some properties of this pseudoindex (see \cite[Lemma 4.7]{KS}). These properties do not depend on $\cJ$.

\begin{Lem}\label{L:pseudo_1_properties}
Let $A,B \in \Sigma$. Then:
\begin{enumerate}[label={(\roman*)}]
	\item if $\gamma^*(A) \neq 0$ then $A \neq \emptyset$;
	\item if $A \subset B$ then $\gamma^*(A) \leq \gamma^*(B)$;
	\item if $h \in \cH(A)$ then $\gamma^*(h(A)) \geq \gamma^*(A)$.
\end{enumerate}
\end{Lem}

We also need the following variant of Benci's pseudoindex (see \cite[Lemma 4.10]{KS}), which additionally satisfies the subadditivity property.
Let $Y \in \Sigma$. We set
\[
	\Sigma_Y:=\left\{A \in \Sigma : A \subset Y \right\}.
\]
\begin{Def}
For any $A \in \Sigma_Y$ we define
\[
	\gamma^*_Y(A):= \inf_{h \in \cH(Y)} \gamma\left(h(A) \cap \cS_{r}^+\right).
\]
\end{Def}
As before, we recall some properties which do not depend on the functional $\cJ$ (see \cite[Lemma 4.10]{KS}).
\begin{Lem}\label{L:pseudo_2_properties}
Let $A,B \in \Sigma_Y$. Then:
\begin{enumerate}[label={(\roman*)}]
	\item $\gamma^*_Y(A) \geq \gamma^*(A)$;
	\item if $A \subset B$ then $\gamma^*_Y(A) \leq \gamma^*_Y(B)$;
	\item if $h \in \cH(Y)$ and $h(A) \subset Y$ then $\gamma^*_Y(h(A)) \geq \gamma^*_Y(A)$;
	\item $\gamma^*_Y(A \cup B) \leq \gamma^*_Y(A) + \gamma(B)$.
\end{enumerate}
\end{Lem}

Kryszewski and Szulkin have proven in \cite{KS} that in $\Sigma$ there are sets of arbitrarily large pseudoindex. Their reasoning is based on the properties of $\cJ$.

\begin{Lem}\label{L:large_pseudoindex}
    For any $k\in \bN$ there exists a set $A_k\in\Sigma$ such that $\gamma^*(A_k)\geq k$.
\end{Lem}
\begin{proof}
    Let $X^+_k$ denote the $k$-dimensional subspace of $X^+$ and put $X_k := X^+_k\oplus X^-$. Due to the assumption (\ref{A7}) and Remark \ref{R:A7} there is a number $R_k\geq r$, where $r$ is defined in (\ref{A5}), such that
    \[
    \sup_{ \substack{v\in X_k \\ \|v\|\geq R_k} } \cJ(v) < \inf_{\|u\|\leq r} \cJ(u).
    \]
    Let us define
    \[
    A_k:=\overline{B_{X_k}(0,R_k)}.
    \]
    The proof that $\gamma^*(A_k)\geq k$ with the properties of $\cJ$ mentioned above is exactly the same as that of \cite[Lemma 4.8]{KS}.

\end{proof}

We are finally in a position to prove the main abstract result of this paper.

\begin{proof}[Proof of Theorem \ref{T:mainTheorem}]
    We define the min-max value
\[
	c_k:=\inf_{\substack{A\in\Sigma \\ \gamma^*(A) \geq k}} \sup_{u \in A} \cJ(u).
\]
Due to Lemma \ref{L:large_pseudoindex} the set $\{A\in\Sigma\,:\,\gamma(A) \geq k\}$ is non-empty and therefore $(c_k)_{k\geq 1}$ is a sequence of real numbers. Moreover
\begin{equation}
\label{E:non_decr_ck}
	b  \leq c_k \leq c_{k+1},
\end{equation} 
for all $k \geq 1$, where $b$ is defined in (\ref{A5}).
Indeed, let $A \in {\Sigma} $ and $h \in \cH(A)$. If $\gamma^*(A) \geq k$, then $\gamma\left(h(A) \cap \cS_{r}^+\right) \geq k$ and, in particular, $h(A) \cap \cS_{r}^+ \neq 0$. Therefore, there exists $u \in A$ such that $h(u) \in \cS_{r}^+$ and $\cJ(u) \geq \cJ(h(u)) \geq \inf_{\cS_{r}^+} \cJ=b$, by the third property of family $\cH(A)$. The second inequality follows from the inclusion $\left\{ A \in {\Sigma} : \gamma(A) \geq k + 1 \right\} \subset \left\{ A \in {\Sigma} : \gamma(A) \geq k \right\}$.

There are two possible cases depending on the maximum of the values $c_k$:
\begin{enumerate}[label={(\roman*)}]
	\item there is an integer $k_0 \geq 1$ such that $c_{k_0} > \beta + 1$;
	\item for all $k \geq 1$, it holds $b \leq c_k \leq \beta + 1$.
\end{enumerate}

We show that both (i) and (ii) cannot hold.

\emph{ Case }(i). Applying Lemma \ref{L:deformationLemma}, thesis $(ii)$, for $\xi>c_{k_0}+1>\beta+2$ and $d=c_{k_0}$ we got $h \in \cH(\cJ^\xi)$ and $\varepsilon>0$ with their properties. From the definition of infimum, there exists an $A \in {\Sigma}$ such that $\gamma^*(A) \geq k_0$ and
\[
	\beta + 1 < \sup_{u \in A} \cJ(u) < c_{k_0} + \varepsilon.
\]
It implies $A\subset \cJ^{c_{k_0}+\varepsilon}\subset \cJ^{\xi}$. Moreover, by the properties of $h$,
\[
h(A)\subset h(\cJ^{c_{k_0}+\varepsilon})\subset \cJ^{c_{k_0}-\varepsilon},
\]
hence
\[
	\sup_{u \in A} \cJ(h(u)) \leq c_{k_0} - \varepsilon.
\]
On the other hand, by property (iii) of pseudoindex $\gamma^*$ (see Lemma \ref{L:pseudo_1_properties} and note that $h(A)\in \Sigma$) we have 
 $\gamma^*(h(A))\geq\gamma^*(A)\geq k_0$, which implies $\sup_{u \in h(A)} \cJ(u)\geq c_{k_0}$, a contradiction.

\emph{Case }(ii). In this case we utilise pseudoindex $\gamma^*_Y$ for $Y=\cJ^{\beta+2}$. The sequence $(c_k)_k$ is convergent since it is bounded by $\beta+1$ and nondecreasing, so let $c:=\lim\limits_{k \to \infty} c_k$. By the definition of the $c_k$'s, we have that $\gamma^*(\cJ^{c+\nu}) \geq k$ for all $1>\nu > 0$ and $k \geq 1$. Note that $\cJ^{c+\nu}\subset Y$.

We define a new sequence $(d_k)$ as
\[
	d_k:= \inf_{  \substack{A\in\Sigma_Y \\ \gamma^*_Y(A)\geq k}  }\sup_{u \in A} \cJ(u).
\]
The numbers $d_k$ are well-defined, since from Lemma \ref{L:pseudo_2_properties}$(i)$ we have
\begin{equation}
\label{E:inf_pseuoind}
	\gamma^*_Y\left(\cJ^{c + \nu}\right) \geq \gamma^*\left(\cJ^{c+\nu}\right) = \infty,
\end{equation}
so the set $\{A\in\Sigma_Y\,:\,\gamma^*_Y(A)>k\}$ is non-empty.

Acting similarly as in \eqref{E:non_decr_ck}, by definition of the sequence $(d_k)$ and by \eqref{E:inf_pseuoind}, we have 
\[
b \leq d_k \leq d_{k+1} \leq c.
\]
Therefore, the sequence $(d_k)_k$ is nondecreasing and bounded, so
\[
	b  \leq d:=\lim_{k \to \infty} d_k \leq c.
\]
For all $\eps > 0$ sufficiently small, $\cJ^{d + \eps} \subset \cJ^{\beta + 2} = Y$, thus $\gamma^*_Y\left(\cJ^{d+\eps}\right)$ is well-defined and arguing as for $\gamma^*\left(\cJ^{c+\nu}\right)$, we have $\gamma^*_Y\left(\cJ^{d+\eps}\right)=\infty$.

By Lemma \ref{L:deformationLemma} applied for $\xi > \beta+2$, and $d\geq b$ constructed above, there exist $\eps > 0$, a symmetric $\tau$-open set $\cN$ with $\gamma(\overline{\cN})=1$, and $h \in {\cH}(Y)$ such that $h(\cJ^{d+\varepsilon}\setminus \cN)\subset \cJ^{d-\varepsilon}$. 
Note that $\cJ^{d+\varepsilon}\setminus \cN,\, \overline{\cN}\cap \cJ^{d+\varepsilon}\in \Sigma_Y$ and 
$\cJ^{d+\eps} = \left(\cJ^{d+\eps} \setminus \cN\right) \cup \left(\overline{\cN} \cap \cJ^{d+\eps}\right)$ 
therefore we can apply Lemma \ref{L:pseudo_2_properties}$(iv)$ to get
\[
\begin{split}
	\infty= \gamma^*_Y\left(\cJ^{d+\eps}\right) &\leq \gamma^*_Y\left(\cJ^{d+\eps} \setminus \cN\right) + \gamma\left(\overline{\cN} \cap \cJ^{d+\eps}\right) \leq \\
    &\leq \gamma^*_Y\left(h(\cJ^{d+\eps} \setminus \cN)\right) + \gamma(\overline{\cN}) \leq \gamma^*_Y\left(\cJ^{d-\eps}\right) + 1,
\end{split}
\]
so $\gamma^*_Y\left(\cJ^{d-\eps}\right)=\infty$. It follows that $d_k \leq d - \eps$ for all $k \geq 1$, but this is a contradiction, since $d_k\to d$. Hence \eqref{F-finite} cannot hold and the proof is completed.
\end{proof}

\section{Applications: Nonlinear Schr\"odinger equations}
\label{sect:appl1}

Consider the following stationary Schr\"odinger equation
\begin{equation}\label{eq:schroedinger}
    -\Delta u + V(x)u = f(u) - \lambda g(u),\quad \mbox{in } \bR^N, \quad N \geq 3.
\end{equation}
When $V$ is $\mathbb{Z}^N$-periodic, \eqref{eq:schroedinger} appears in nonlinear optics, where photonic crystals admitting nonlinear effects are studied, see \cite{Kuchment, Pankov}. In this case, it describes the propagation of solitons being solitary wave solutions $\Psi(t,u) = u(x) e^{-i \omega t}$ to the time-dependent, nonlinear Schr\"odinger equation.

Then, the term $ f(u)-\lambda g(u)$ describes the nonlinear part of the polarization of the medium, e.g. in self-focusing Kerr-like media one has $f(u)= |u|^2 u$ and $g(u)=0$, see \cite{BURYAK200263, Slusher}. In the case, when $g \not\equiv 0$, e.g. $f(u) = |u|^{p-2}u$, $g(u) = |u|^{q-2}u$, $2 < q < p$, we deal with a mixture of self-focusing and defocusing materials.

Suppose that the external potential $V \in L^\infty (\R^N)$ is $\mathbb{Z}^N$-periodic. Then it is known that the spectrum of the operator $\mathcal{A} := -\Delta + V(x)$ on $L^2 (\R^N)$ consists of closed, pairwise disjoint intervals (\cite[Theorem XIII. 100]{ReedSimon}) and is unbounded from above. Hence, we assume the following
\begin{enumerate}
[label=(V\arabic{*}),ref=V\arabic{*}]\setcounter{enumi}{0}
    \item \label{V1} $V \in L^\infty (\R^N)$ is $\mathbb{Z}^N$-periodic,
    $$
    0 \not\in \sigma (\cA) \quad \mbox{and} \quad \inf \sigma(\cA) < 0.
    $$
\end{enumerate}
In such a case, we say that $0$ lies in the spectral gap of $\cA$. Note that, in particular, $V$ cannot be constant, since for $V \equiv V_0 \in \R$ we have $\sigma(\cA) = \sigma(-\Delta + V(x)) = [V_0, \infty)$ and therefore spectral gaps do not exist.

Here, in addition, we assume that $\inf \sigma(\cA) < 0$, since the multiplicity of solutions in the positive-definite case $\inf \sigma(\cA) > 0$ has been shown in \cite[Theorem 1.2]{BB_JMAA}. If $g(u) = 0$, the existence of ground states has been widely studied by many authors, e.g. \cite{MR2271695, MR2957647, MR3494969, MR3551463, KS, Pankov, MR1162728, SzW} and the references therein.

In what follows, we use $\lesssim$ to denote the inequality up to a positive multiplicative constant. We impose the following on the nonlinear functions $f$ and $g$.
\begin{enumerate}
[label=(F\arabic{*}),ref=F\arabic{*}]\setcounter{enumi}{0}
\item \label{F1} $f: \R \to \R$ is odd, continuous and there is $2<p<2^* := \frac{2N}{N-2}$ such that
\[
	|f(u)| \lesssim 1+|u|^{p-1} \text{ for all } u \in \R.
\]
\item \label{F2} $f(u)=o(|u|)$ as $u \to 0$.
\item \label{F3} There is $2 < q < p$ such that $F(u)/|u|^q \to \infty$ as $|u| \to \infty$, where $F(u)=\int_0^u f(s) \, dx$ and $F(u) \geq 0$ for all $u \in \R$.
\item \label{F4} $u \mapsto f(u)/|u|^{q-1}$ is nondecreasing in $(-\infty,0)$ and on $(0,\infty)$.
\item \label{F5} There is $\rho > 0$ such that $|u|^{p-1} \lesssim |f(u)| \lesssim |u|^{p-1}$ for $|u| \geq \rho$.
\end{enumerate}
\begin{enumerate}
[label=(G\arabic{*}),ref=G\arabic{*}]\setcounter{enumi}{0}
\item \label{G1} $g: \R \to \R$ is odd, continuous such that
\[
	|g(u)| \lesssim 1+|u|^{q-1} \text{ for all } u \in \R.
\]
\item \label{G2} $g(u)=o(|u|)$ as $u \to 0$.
\item \label{G3} $u \mapsto g(u)/|u|^{q-1}$ is nonincreasing in $(-\infty,0)$ and on $(0,\infty)$ and there holds
\[
	g(u)u \geq 0 \text{ for all } u \in \R.
\]
\end{enumerate}
For examples of $f$ and $g$ satisfying the foregoing assumptions, we refer to \cite{BB}.

It is classical to check that under (\ref{F1}), (\ref{F2}), (\ref{G1}), (\ref{G2}) the \textit{energy functional} $\cJ : H^1 (\R^N) \rightarrow \R$ given by
$$
\cJ(u) := \frac12 \int_{\R^N} |\nabla u|^2 + V(x)u^2 \, dx - \int_{\R^N} F(u) \, dx + \lambda \int_{\R^N} G(u) \, dx,
$$
where $G(u) = \int_0^u g(s) \, ds$, is of $\cC^1$-class and its critical points correspond to weak solutions to \eqref{eq:schroedinger}. Since $0$ lies in the spectral gap of $\cA$, there is an orthogonal decomposition of the space $X := H^1(\R^N) = X^+ \oplus X^-$, which corresponds to the decomposition of the spectrum $\sigma(A)$ into positive and negative parts, such that the quadratic form
$$
u \mapsto \int_{\R^N} |\nabla u|^2 + V(x)u^2 \, dx
$$
is positive definite on $X^+$ and negative definite on $X^-$. Hence, if $u = u^+ + u^- \in X^+ \oplus X^-$, with $u^\pm \in X^\pm$, we may introduce the norm, equivalent to the standard one on $H^1(\bR^N)$, as follows
$$
\|u^\pm\|^2 := \pm \int_{\R^N} |\nabla u|^2 + V(x)u^2 \, dx
$$
and $\|u\|^2 := \|u^+\|^2 + \|u^-\|^2$. Then $\cJ$ can be rewritten as
$$
\cJ(u) = \frac12 \|u^+\|^2 - \frac12 \|u^-\|^2 - \cI(u),
$$
where $\cI(u) = \int_{\R^N} F(u) - \lambda G(u) \, dx$. Moreover, from the assumption (\ref{V1}), there is $\mu_0 > 0$ such that
\begin{equation}\label{E:mu0}
    \mu_0 |u|_2 \leq \|u\| \quad \mbox{ for } u \in X.
\end{equation}

\subsection{Verification of (\ref{A1})--(\ref{A8})}

Note that, thanks to Example \ref{ex:1}, $(H^1(\R^N), \mathbb{Z}^N)$ is a dislocation space with discreteness property. It is well known that the derivative $\cI'$ is sequentially weak-to-weak* continuous.

\begin{Lem}
$\cI'$ satisfies (\ref{GWC}).
\end{Lem}

\begin{proof}
To verify (\ref{GWC}), take any bounded sequences $(v_n), (\varphi_n) \subset H^1(\R^N)$ with $\varphi_n \stackrel{\Z^N}{\weakto} 0$. Thanks to Example \ref{ex:1}(ii) it means that $\varphi_n \to 0$ in $L^p (\R^N)$ and in $L^q (\R^N)$. Thus, thanks to (\ref{F1}), (\ref{F2}), (\ref{G1}), (\ref{G2}) for every $\varepsilon > 0$ we find $C_\varepsilon > 0$ such that
\begin{align*}
|\cI'(v_n)(\varphi_n)| &= \left| \int_{\R^N} f(v_n) \varphi_n - \lambda g(v_n) \varphi_n \, dx \right|\\
&\leq \varepsilon \int_{\R^N} |v_n| | \varphi_n | \, dx + C_\varepsilon \int_{\R^N} ( |v_n|^{p-1} + |v_n|^{q-1} ) |\varphi_n| \, dx \\
&\leq \varepsilon |v_n|_2 |\varphi_n|_2 + C_\varepsilon |v_n|_{p}^{p-1} |\varphi_n|_p + C_\varepsilon |v_n|_{q}^{q-1} |\varphi_n|_q \\
&\lesssim \varepsilon + |\varphi_n|_p + |\varphi_n|_q.
\end{align*}
Hence
$$
\limsup_{n\to\infty} |\cI'(v_n)(\varphi_n)| \lesssim \varepsilon
$$
for every $\varepsilon > 0$ and therefore $\cI'(v_n)(\varphi_n) \to 0$.
\end{proof}

It is clear that $\cJ$ is even and $\mathbb{Z}^N$-invariant and $X^-, \, X^+$ are $\bZ^N$-invariant. Hence (\ref{A1})--(\ref{A3}) are satisfied. Thanks to \cite[Theorem 1.2]{BB} we know that (\ref{A4}) is satisfied for sufficiently small $\lambda > 0$ and $\rho > 0$.

\begin{Lem}\label{lem:mp}
There is $r > 0$ such that
$$
b := \inf_{\cS_r^+} \cJ > 0.
$$
\end{Lem}

\begin{proof}
Fix $u^+ \in X^+$ and note that, by the continuity of Sobolev embeddings, 
\begin{align*}
\cJ(u^+) \geq \frac12 \|u^+\|^2 - \int_{\R^N} F(u^+) \, dx \geq \frac12 \|u^+\|^2 - \varepsilon C \|u^+\|^2 - \widetilde{C_\varepsilon} \|u^+\|^{p}
\end{align*}
for some $C, \widetilde{C_\eps} > 0$. Choosing  $\varepsilon = \frac{1}{4C}$ and sufficiently small $r > 0$ we easily obtain that
$$
b = \inf_{\cS_r^+} \cJ \geq  \frac{r^2}{8} > 0.
$$
\end{proof}

\begin{Lem}\label{lem:sup}
Suppose that $\lambda > 0$ is sufficiently small. There is radius $r_0 > 0$ such that
$$
\sup\{ \cJ(u) \ : \ u \in X, \ \|u^+\|<r_0 \} < b.
$$
\end{Lem}

\begin{proof}
We recall that, thanks to (\ref{F1})--(\ref{F3}) and (\ref{G1})--(\ref{G2}), for every $\varepsilon > 0$ we can find $C_{F,\varepsilon}, C_{G,\varepsilon} > 0$ such that
\begin{equation}\label{C-eps}
G(u) \leq \varepsilon u^2 + C_{G,\varepsilon} |u|^q, \quad F(u) \geq C_{F,\varepsilon} |u|^q - \varepsilon u^2, \quad C_{G,\varepsilon} \geq C_{F,\varepsilon}.
\end{equation}
Then, for $\varepsilon = \frac{\mu_0}{4}$ and $\lambda \leq \frac{C_{F,\varepsilon}}{C_{G,\varepsilon}}$ and, using \eqref{E:mu0},
\begin{align*}
\cJ(u) &\leq \frac12 \|u^+\|^2 - \frac12 \|u^-\|^2 -  (C_{F,\varepsilon}-\lambda C_{G,\varepsilon}) |u|_q^q + \varepsilon (1+\lambda) |u|_2^2  \\
&\leq \frac12 \|u^+\|^2 - \frac12 \|u^-\|^2 + 2\varepsilon  |u|_2^2 \leq \left( \frac12 + \frac{2\varepsilon}{\mu_0} \right) \|u^+\|^2 - \left( \frac12 - \frac{2\varepsilon }{\mu_0} \right) \|u^-\|^2 = \|u^+\|^2. 
\end{align*}
Hence it is sufficient to take $r_0 = \frac{\sqrt{b}}{2}$.
\end{proof}

\begin{Lem}\label{lem:Jinfty}
Assume that $\lambda > 0$ is sufficiently small. Suppose that $(u_n) \subset X$ is so that $\|u_n^-\| \to \infty$ and $\|u_n^+\|$ stays bounded. Then $\cJ(u_n) \to - \infty$.
\end{Lem}

\begin{proof}
Let $(u_n)$ be a sequence as in the statement. Then, repeating the computation from the proof of Lemma \ref{lem:sup}, we get that for every $\varepsilon > 0$ and for $\lambda \in \left( 0, \frac{C_{F,\varepsilon}}{C_{G,\varepsilon}} \right]$,
$$
\cJ(u_n) \leq \left( \frac12 + \frac{2\varepsilon}{\mu_0} \right) \|u^+\|^2 - \left( \frac12 - \frac{2\varepsilon }{\mu_0} \right) \|u^-\|^2.
$$
Choosing $\varepsilon = \frac{\mu_0}{8}$ we get
$$
\cJ(u_n) \leq \frac34 \|u_n^+\|^2 - \frac14 \|u_n^-\|^2 \to -\infty.
$$
\end{proof}

From Lemmas \ref{lem:mp}, \ref{lem:sup}, and \ref{lem:Jinfty} we see that (\ref{A5}), (\ref{A6}), and (\ref{AX}) are satisfied for sufficiently small $\lambda$.

\begin{Lem}
Assume that $\lambda > 0$ is sufficiently small. (\ref{A7}) holds.
\end{Lem}

\begin{proof}
Let $W \subset X^+$ be a finite-dimensional subspace and let $(u_n^+) \subset W$. Note that on $W$ all norms are equivalent. Then, using \eqref{C-eps}, for $\lambda \in \left(0,\frac{C_{F,\varepsilon}}{C_{G,\varepsilon}}\right)$
$$
\cI(u_n^+) = \int_{\R^N} F(u_n^+) - \lambda G(u_n^+) \, dx \geq  (C_{F,\varepsilon}-\lambda C_{G,\varepsilon}) |u_n^+|_q^q - \varepsilon (1+\lambda) |u_n^+|_2^2 \gtrsim \|u_n^+\|^q - \|u_n^+\|^2,
$$
and therefore $\frac{\cI(u_n^+)}{\|u_n^+\|^2} \to \infty$, since $q>2$.
\end{proof}

\begin{Lem}
Assume that $\lambda > 0$ and $\rho > 0$ in (\ref{F5}) are sufficiently small. Let $(u_n) \subset X$ satisfy
$$
\cJ(u_n) \leq \beta, \quad \cJ'(u_n) \to 0
$$
for some $\beta \in \R$. Then $(u_n)$ is bounded in $X$.
\end{Lem}

\begin{proof}
We will follow the argument from \cite[Lemma 5.1]{BB}. Suppose by contradiction that $\|u_n\| \to \infty$. Note that for sufficiently large $n$
\begin{equation}\label{E:Jproperty}
    | \cJ'(u_n)(u_n) | \leq \| \cJ'(u_n) \| \|u_n\| \leq \frac12 \|u_n\|.
\end{equation}
Thus
$$
\|u_n\|^2 = \|u_n^+\|^2 + \|u_n^-\|^2 \leq \int_{\R^N} \left( f(u_n) - \lambda g(u_n) \right) (u_n^+ - u_n^-) \, dx + \frac12 \|u_n\|.
$$
Then we write
$$
\int_{\R^N} \left( f(u_n) - \lambda g(u_n) \right) (u_n^+ - u_n^-) \, dx = I_1 + I_2,
$$
where
\begin{align*}
I_1 &:= \int_{|u_n| < \rho} \left( f(u_n) - \lambda g(u_n) \right) (u_n^+ - u_n^-) \, dx, \\
I_2 &:= \int_{|u_n| \geq \rho} \left( f(u_n) - \lambda g(u_n) \right) (u_n^+ - u_n^-) \, dx.
\end{align*}
Repeating the same argument as in \cite[Lemma 5.1]{BB} we obtain that for every $\varepsilon > 0$ there is $C_\varepsilon$ such that
\begin{equation}\label{eq:I1}
I_1 \lesssim \left((1+\lambda)\varepsilon + C_\varepsilon \rho^{p-2} + \lambda C_\varepsilon \rho^{q-2} \right) \|u_n\|^2.
\end{equation}
Moreover, for $I_2$, we get the following estimate
\begin{equation}\label{eq:I2}
I_2 \lesssim \left(1 + \lambda \frac{g(\rho)}{f(\rho)} \right) |u_n|_p^p.
\end{equation}
Then we note that, by \eqref{E:Jproperty},
$$
\beta + \frac14 \|u_n\| \geq \cJ(u_n) - \frac12 \cJ'(u_n)(u_n) = \int_{\R^N} \Phi(u_n) \, dx
$$
with $\Phi(u) = \frac12 f(u)u - F(u) + \lambda G(u) - \frac{\lambda}{2} g(u)u$. Next, we arrive at
$$
\beta + \frac14 \|u_n\| + \int_{|u_n| < \rho} |\Phi(u_n)| \, dx \gtrsim \left(1 - \lambda \frac{g(\rho)}{f(\rho)} \right) \int_{|u_n| \geq \rho} |u_n|^p \,dx.
$$
Thus
$$
\int_{|u_n| \geq \rho} |u_n|^p \,dx \leq C \left(1 - \lambda \frac{g(\rho)}{f(\rho)} \right)^{-1} \left( \beta + \frac14 \|u_n\| + \int_{|u_n| < \rho} |\Phi(u_n)| \, dx \right)
$$
for some $C > 0$. Hence, from \eqref{eq:I2},
\begin{align*}
I_2 &\leq D(\lambda, \rho) \left( \int_{|u_n| < \rho} |u_n|^p \,dx + \int_{|u_n| \geq \rho} |u_n|^p \,dx \right) \\
&\leq D(\lambda,\rho) \left( \frac{\rho^{p-2}}{\mu_0} + \frac{C}{\left(  1 - \lambda \frac{g(\rho)}{f(\rho)}\right) \mu_0} \sup_{|t| \leq \rho} \frac{|\Phi(t)|}{t^2} \right) \|u_n\|^2 + \widetilde{C} (1 + \|u_n\|)
\end{align*}
with $D(\lambda, \rho)$ defined as in the proof of \cite[Lemma 5.1]{BB} and some $\widetilde{C} = \widetilde{C}(\lambda, \rho, \varepsilon) > 0$. 
Finally, taking into account also \eqref{eq:I1},
$$
\|u_n\|^2 \leq I_1 + I_2 + \frac12 \|u_n\| \leq \frac{K}{\mu_0} \|u_n\|^2 +\widetilde{C} (1 + \|u_n\|)
$$
and, as in \cite[Lemma 5.1]{BB}, $K < \mu_0$ for sufficiently small $\rho $ and $\lambda$, which gives a contradiction.
\end{proof}

Now, repeating the argument of \cite[Proposition 5.2]{BB}, we see that (\ref{A8}) is satisfied. As a consequence of Theorem \ref{T:mainTheorem}, we obtain the following result.

\begin{Th}
Suppose that (\ref{V1}), (\ref{F1})--(\ref{F5}), (\ref{G1})--(\ref{G3}) hold. If $\lambda > 0$ and $\rho > 0$ are sufficiently small, there are infinitely many pairs $\pm u$ of geometrically distinct solutions to \eqref{eq:schroedinger}.
\end{Th}

\section{Applications: Time-harmonic, electromagnetic waves}
\label{sect:appl2}

We consider the system of Maxwell equations of the form
$$
\left\{ \begin{array}{l}
\curl \cH = \cJ + \frac{\partial \cD}{\partial t} \\
\div (\cD) = \rho \\
\frac{\partial \cB}{\partial t} + \curl \cE = 0 \\
\div (\cB) = 0,
\end{array} \right.
$$
where $\cE$ is the electric field, $\cB$ is the magnetic field, $\cD$ is the electric displacement field, $\cH$ denotes the magnetic induction, $\cJ$ the electric current intensity, and $\rho$ the electric charge density. We also have the following constitutive relations
$$
\begin{cases}
\cD=\eps \cE + \cP\\
\cH=\frac{1}{\mu}\cB - \cM,
\end{cases}
$$
where $\cP$ is the polarization and $\cM$ is the magnetization. In the absence of charges, currents, and magnetization, and assuming that $\mu \equiv 1$, where $\mu$ is the permeability of the medium, we obtain the time-dependent, electromagnetic wave equation (see e.g. \cite{BaMe1})
$$
\curl \left( \curl \cE \right) + \varepsilon \frac{\partial^2 \cE}{\partial t^2} = - \frac{\partial^2 \cP}{\partial t^2},
$$
where $\varepsilon$ is the permittivity of the medium. We look for a time-harmonic field $\cE = \mathbf{E}(x) \cos(\omega t)$ and suppose that the nonlinear polarization $\cP$ is of the form 
$$
\cP = \chi\left( \frac12 |\mathbf{E}|^2 \right) \mathbf{E} \cos(\omega t), 
$$
i.e. the scalar dielectric susceptibility $\chi$ depends only on the time average 
$$
\frac{1}{T} \int_0^T | \cE(x,t)|^2 \, dt = \frac12 |\mathbf{E}|^2
$$ 
of the intensity of the electric field, where $T := \frac{2\pi}{\omega}$. Hence, $\cP = \mathbf{P}(\mathbf{E}(x)) \cos(\omega t)$, where $\mathbf{P}(\mathbf{E}) := \chi\left( \frac12 |\mathbf{E}|^2 \right) \mathbf{E}$. This ansatz leads to 
\begin{equation}\label{eq:maxwell}
\curl (\curl \mathbf{E}) + V(x)\mathbf{E} = h(\mathbf{E}), \quad x \in \R^3
\end{equation}
with $V(x) := - \omega^2 \varepsilon(x)$ and $h(\mathbf{E}(x)) := \mathbf{P}(\mathbf{E}(x)) \omega^2$. For media with Kerr effect, strong electric fields $\cE$ of high intensity cause the refractive index to vary quadratically, and then $\cP$ is of the form
$$
\cP( t, x) = \frac{\alpha(x)}{2} |\mathbf{E}|^2 \mathbf{E} \cos(\omega t),
$$
see \cite{Nie, S1}. Assuming that $\alpha(x) \equiv \alpha$ is a constant, we get $\mathbf{P}(\mathbf{E}(x)) = \frac{\alpha}{2} | \mathbf{E}(x)|^2 \mathbf{E}(x)$. In this example, we are interested in the more general case, where the polarization may consist of two competing terms, e.g. $\mathbf{P}(\mathbf{E}) = |\mathbf{E}|^{p-2} \mathbf{E} - |\mathbf{E}|^{q-2} \mathbf{E}$.

Looking for classical solutions to \eqref{eq:maxwell} of the form (see e.g. \cite{BDPR, Z})
\begin{equation}\label{eq:form}
\mathbf{E}(x) = \frac{u(r, x_3)}{r} \left( \begin{array}{c} -x_2 \\ x_1 \\ 0 \end{array} \right), \quad r = \sqrt{x_1^2+x_2^2}
\end{equation}
leads to the following equation with a singular term
\begin{equation}\label{eq:schroedinger2}
-\Delta u + V(x)u + \frac{a}{r^2} u = f(u) - \lambda g(u), \quad x = (y,z) \in \R^K \times \R^{N-K}, \ r = |y|,
\end{equation}
with $N=3$, $K=2$, $a=1$, where $\Delta = \frac{\partial ^2}{\partial r^2} + \frac{1}{r} \frac{\partial}{\partial r} + \frac{\partial^2}{\partial x_3^2}$ is the 3-dimensional Laplacian operator in cylindrically symmetric coordinates $(r, x_3)$, and nonlinear terms are described by the following relation
\begin{equation}\label{eq:h}
h(\mathbf{E}) = f(\alpha)w - \lambda g(\alpha)w,
\end{equation}
where $\mathbf{E} = \alpha w$ for some $w \in \R^3$, $|w|=1$, $\alpha \in \R$ and $h$ is the nonlinear term in \eqref{eq:maxwell}. This equivalence also holds for weak solutions (see \cite{Bi, GMS}). We note that equations of the form \eqref{eq:schroedinger2} with $V \equiv 0$ and $g \equiv 0$ have been studied in \cite{BBR, BMS}. 

It is known that in this setting, the total electromagnetic energy
$$
\cL(t) := \frac12 \int_{\R^3} \cE \cD + \cB \cH \, dx
$$
is finite and constant (does not depend on $t$), see \cite[Proposition 6.3]{BB}. To study the more general equation \eqref{eq:schroedinger2}, we assume $N > K \geq 2$, $a \in \R$ and
\begin{enumerate}
[label=(V\arabic{*}),ref=V\arabic{*}]\setcounter{enumi}{1}
    \item \label{V2} $V \in L^\infty (\R^N)$ is $\cO(K) \times \{id_{N-K}\}$ invariant and $\mathbb{Z}^{N-K}$-periodic in $z$,
    $$
    0 \not\in \sigma \left( -\Delta + \frac{a}{r^2} + V(x) \right)  \quad \mbox{and} \quad \inf \left( -\Delta + \frac{a}{r^2} + V(x) \right) < 0.
    $$
\end{enumerate}
We will use the same assumptions (\ref{F1})--(\ref{F5}), (\ref{G1})--(\ref{G3}) as in Section \ref{sect:appl1}. 

Let $\cG(K) := \cO(K) \times \{id_{N-K}\}$ and recall that $H^1_{\cG(K)} (\R^N)$ denotes the subspace of $H^1 (\R^N)$ consisting of $\cG(K)$-invariant functions. We introduce the space 
$$
X := \left\{ u \in H^1_{\cG(K)} (\R^N) \ : \  \int_{\R^N} \frac{u^2}{r^2} \, dx < \infty \right\},
$$
which admits an orthogonal splitting $X = X^+ \oplus X^-$ corresponding to the decomposition of the spectrum $\sigma \left( -\Delta + \frac{a}{r^2} + V(x) \right)$ into its positive and negative parts. Then we introduce the norm in $X^\pm$ as
$$
\|u^\pm \|^2 := \pm \int_{\R^N} |\nabla u|^2 + a \frac{u^2}{r^2} + V(x) u^2 \, dx, \quad u^\pm \in X^\pm,
$$
and we put $\|u\|^2 := \|u^+\|^2 + \|u^-\|^2$ for $u = u^+ + u^- \in X$. Hence, the energy functional associated with \eqref{eq:schroedinger2} is given by
$$
\cJ(u) := \frac12 \|u^+\|^2 - \frac12 \|u^-\|^2 - \int_{\R^N} F(u) \, dx + \lambda \int_{\R^N} G(u) \, dx, \quad u \in X.
$$
For $K > 2$ we have the following Hardy inequality (see \cite{BT})
$$
\int_{\R^N} \frac{u^2}{r^2} \, dx \leq \left( \frac{2}{K-2} \right)^2 \int_{\R^N} |\nabla u|^2 \, dx, \quad u \in H^1 (\R^N);
$$
therefore, $\int_{\R^N} \frac{u^2}{r^2} \, dx < \infty$ for every $u \in H^1 (\R^N)$, and in particular for every test function $u \in \cC_0^\infty (\R^N)$. Hence, thanks to the Palais' principle of symmetric criticality, we know that critical points of $\cJ$ are \textit{weak solutions} to \eqref{eq:schroedinger2}. For $K = 2$ it is not true that $\cC_0^\infty (\R^N) \subset \left\{ u \in H^1 (\R^N) \ : \ \int_{\R^N} \frac{u^2}{r^2} \, dx < \infty \right\}$ so, in this case, we say that $u \in X$ is a weak solution if it is a critical point of $\cJ$.

Thanks to Example \ref{ex:2} we know that $(H^1_{\cG(K)}, \Z^{N-K})$ is a dislocation space with discreteness property, and one can check that $(X, \Z^{N-K})$ is such a space as well. Then, as in Section \ref{sect:appl1}, it can be shown (\ref{A1})--(\ref{A8}) are satisfied, assuming that $\lambda > 0$ and $\rho > 0$ in (\ref{F5}) are small enough. Hence, we obtain the following result.

\begin{Th}\label{th:shroed2}
Suppose that (\ref{V2}), (\ref{F1})--(\ref{F5}), (\ref{G1})--(\ref{G3}) hold. If $\lambda > 0$ and $\rho > 0$ are sufficiently small, there are infinitely many pairs $\pm u \in X$ of geometrically distinct solutions to \eqref{eq:schroedinger2}.
\end{Th}

To apply the foregoing theorem to the curl-curl problem \eqref{eq:maxwell}, we define the energy functional $\cJ_{curl} : H^1 (\R^3; \R^3) \rightarrow \R$ by
$$
\cJ_{curl} (\mathbf{E}) := \frac12 \int_{\R^3} | \nabla \times \mathbf{E} |^2 + V(x) |\mathbf{E}|^2 \, dx - \int_{\R^N} H(\mathbf{E}) \, dx,
$$
where $H(\mathbf{E}) := \int_0^1 h(t\mathbf{E}) \cdot \mathbf{E} \, dt$. Then $\cJ_{curl}$ is of $\cC^1$-class and its critical points are weak solutions to \eqref{eq:maxwell}. From \cite[Theorem 1.1]{Bi}, under our assumptions, we have the following.

\begin{Prop}\label{prop:equiv}
Let $N = 3$, $K=2$, $a=1$. If $\mathbf{E} \in H^1 (\R^3; \R^3)$ is a weak solution to \eqref{eq:maxwell}, with $h$ given by \eqref{eq:h}, of the form \eqref{eq:form} with some cylindrically symmetric $u$, then $u \in X$ is a weak solution to \eqref{eq:schroedinger2}. If $u \in X$ is a weak solution to \eqref{eq:schroedinger2}, then $\mathbf{E}$ given by the formula \eqref{eq:form} belongs to $H^1 (\R^3; \R^3)$ and is a weak solution to \eqref{eq:maxwell}. In addition, $\div \mathbf{E} = 0$ and $\cJ_{curl} (\mathbf{E}) = \cJ(u)$.
\end{Prop}

We note the following fact.

\begin{Lem}\label{lem:orbitsEquiv}
Let $N = 3$, $K=2$, $a=1$. Solutions $u_1, u_2 \in X$ are geometrically distinct if and only if $\mathbf{E}_1, \mathbf{E}_2 \in H^1(\R^3; \R^3)$ are geometrically distinct.
\end{Lem}

\begin{proof}
Suppose that $\cO(u_1) = \cO(u_2)$. Then $u_1(r,x_3) = u_2 (r, x_3 + z)$ for some $z \in \Z$ and a.e. $r > 0$, $x_3 \in \R$. Then clearly
$$
\mathbf{E}_1(r, x_3) = \frac{u_1(r, x_3)}{r} \left( \begin{array}{c} -x_2 \\ x_1 \\ 0 \end{array} \right) = \frac{u_2(r, x_3 + z)}{r} \left( \begin{array}{c} -x_2 \\ x_1 \\ 0 \end{array} \right) = \mathbf{E}_1(r, x_3 + z)
$$
and $\cO(\mathbf{E_1}) = \cO(\mathbf{E}_2)$. On the other hand, suppose that $\mathbf{E}_1, \mathbf{E}_2$ are of the form \eqref{eq:form} and $\cO(\mathbf{E_1}) = \cO(\mathbf{E}_2)$, which means that
$$
\frac{u_1(r, x_3)}{r} \left( \begin{array}{c} -x_2 \\ x_1 \\ 0 \end{array} \right) = \frac{u_2(r, x_3 + z)}{r} \left( \begin{array}{c} -x_2 \\ x_1 \\ 0 \end{array} \right)
$$
for some $z \in \Z$ and a.e. $r > 0$, $x_3 \in \R$. In particular 
\begin{align*}
\left\{ \begin{array}{l}
-x_2 u_1 (r, x_3) = -x_2 u_2 (r, x_3 + z), \\
x_1 u_1 (r, x_3) = x_1 u_2 (r, x_3 + z).
\end{array} \right.
\end{align*}
Thus $u_1(r,x_3) = u_2 (r, x_3 + z)$ for a.e. $r > 0$, $x_3 \in \R$. Hence $\cO(u_1) = \cO(u_2)$.
\end{proof}

As a straightforward consequence of Theorem \ref{th:shroed2}, Proposition \ref{prop:equiv}, and Lemma \ref{lem:orbitsEquiv}, we obtain the following result.

\begin{Th}
Suppose that (\ref{V2}), (\ref{F1})--(\ref{F5}), (\ref{G1})--(\ref{G3}) hold. If $\lambda > 0$ and $\rho > 0$ are sufficiently small, there are infinitely many pairs $\pm \mathbf{E} \in H^1(\R^3; \R^3)$ of the form \eqref{eq:form}, in particular with $\div \mathbf{E} = 0$, of geometrically distinct solutions to \eqref{eq:maxwell} with $h$ given by \eqref{eq:h}.
\end{Th}

\appendix

\section{Existence of critical points - an abstract formulation}
\label{appendix}

Suppose that, in lieu of (\ref{A4}), we assume that 
\begin{enumerate}
[label=(A4'),ref=A4']
    \item\label{A4N} there are $\delta > 0$, $\rho > 0$, and a nonempty set $\cP \subset X \setminus X^-$ such that for every $u \in \cP$ there is radius $R = R(u) > \rho$ with
    $$
    \inf_{\cS_\rho^+} \cJ > \max \left\{ \sup_{\partial M(u)} \cJ, \sup_{\triple{v} \leq \delta} \cJ(v) \right\},
    $$
    where  
    $$
    M(u) := \{ tu + v^- \ : \ v^- \in X^-, \ t \geq 0, \ \|tu+v^-\| \leq R \}.
    $$
\end{enumerate}
Moreover, instead of (\ref{A8}), we assume a bit stronger condition:
\begin{enumerate}
[label=(A9'),ref=A9']
    \item\label{A9N} for any $\beta \in \R$, there is $M_\beta > 0$ such that any sequence $(u_n) \subset X$ with the properties
    $$
    \limsup_{n\to\infty} \cJ(u_n) \leq \beta, \quad (1+\|u_n\|) \cJ'(u_n) \to 0
    $$
    satisfies
    $$
    \limsup_{n\to\infty} \|u_n\| \leq M_\beta.
    $$
\end{enumerate}
Then the following holds true.

\begin{Th} \label{Cor:appendix}
Suppose that $(X, G)$ is a dislocation space with $G$ acting unitarily on $X$. Suppose that $\cJ$ is a nonlinear functional of the form \eqref{E:introJ} with $X = X^+ \oplus X^-$, where $X^\pm$ are $G$-invariant. If $\cI$ is $G$-invariant, $\cC^1$-class with $\cI(0) = 0$, $\cI'$ being sequentially weak-to-weak* continuous satisfying (\ref{GWC}) and (\ref{A4N}), (\ref{A9N}) hold, then there exists a nontrivial critical point of $\cJ$.
\end{Th}

\begin{proof}
From \cite[Theorem 2.1]{BB} there exists a sequence $(u_n) \subset X$ such that
$$
\sup_n \cJ(u_n) \leq c, \quad (1+\|u_n\|) \cJ'(u_n) \to 0 \mbox{ in } X^*, \quad \inf_n \triple{u_n} \geq \frac{\delta}{2}
$$
for some $c > 0$. Thanks to (\ref{A9N}), $(u_n)$ is bounded. Suppose that $u_n \stackrel{G}{\weakto} 0$, namely $\sup_{g \in G} \langle u_n, g \varphi \rangle \to 0$ for any $\varphi \in X$. Since $\cJ'(u_n)(u_n^\pm) \to 0$, using Proposition \ref{P:D-convergence_of_projections} and (\ref{GWC}), we get
$$
\|u_n^{\pm}\|^2 = \pm \cI'(u_n)(u_n^{\pm}) + o(1) \to 0.
$$
This is a contradiction with $\delta / 2 \leq \triple{u_n} \leq \|u_n\| \to 0$. Hence, up to choosing a subsequence, we can find $g_n \in G$ such that $v_n := g_n u_n \weakto v \neq 0$. Therefore, $\| \nabla \cJ (v_n) \| = \| \nabla \cJ(g_n u_n) \| = \| g_n \nabla \cJ(u_n)\| = \|\nabla \cJ(u_n)\| \to 0$, since $G$ acts on $X$ by isometries. Thus $\cJ'(v_n) \to 0$ and from the weak-to-weak* continuity of $\cJ'$ we get that $\cJ'(v) = 0$, so $\critJ \setminus \{0\} \neq \emptyset$.
\end{proof}

Clearly, under the assumptions of Theorem \ref{Cor:appendix}, there exists a nontrivial $G$-orbit of critical points of $\cJ$.
Combining Theorem \ref{Cor:appendix} and Theorem \ref{T:mainTheorem}, we obtain the following result.

\begin{Th}
Under the assumptions (\ref{A1})--(\ref{A3}), (\ref{A4N}), (\ref{A5})--(\ref{A7}), (\ref{A9N}) the functional $\cJ$ has infinitely many geometrically distinct critical points.
\end{Th}

\section*{Acknowledgements}
We would like to thank the Referee for the thorough review of the article and the valuable comments and suggestions.

This work was conducted during Federico Bernini’s affiliation with the University of Milan, whose support he gratefully acknowledges.
Federico Bernini is a member of GNAMPA (INdAM) and is partially supported by GNAMPA project 2024 {\em Problemi spettrali e di ottimizzazione di forma: aspetti qualitativi e stime quantitative}.
Bartosz Bieganowski and Daniel Strzelecki were partly supported by the National Science Centre, Poland (Grant no. 2022/47/D/ST1/00487). This work was partly supported by the Thematic Research Programme “Variational and geometrical methods in partial differential equations”, University of Warsaw, Excellence Initiative - Research University.

\section*{Statements and Declarations}
\textbf{Conflict of interest.} On behalf of all authors, the corresponding author states that there is no conflict of interest.

\textbf{Data availability.} Not applicable.

\bibliographystyle{abbrv}
\bibliography{Bibliography}

\end{document}